\documentclass {amsart}

\usepackage[latin1]{inputenc}%
\usepackage{amssymb,amsfonts,epsfig,amsmath,graphics,graphicx}%
\usepackage{amsthm}%
\usepackage[all]{xy}%
\usepackage{verbatim}%
\usepackage{enumitem}
\usepackage{bm}

\newtheorem{theorem} {Theorem}[section]
\newtheorem{proposition}[theorem]{Proposition}
\newtheorem{lemma}[theorem]{Lemma}
\newtheorem{corollary}[theorem]{Corollary}
\newtheorem{definition}[theorem]{Definition}
\newtheorem{question}[theorem]{Question}

\theoremstyle{plain}
\newtheorem{theoremint} {Theorem}
\newtheorem*{theoremint2} {Theorem 1}
\newtheorem*{definitionint}{Definition}

\newtheorem{Thmbis} {Theorem}

\theoremstyle{remark}
\newtheorem{remark}[theorem]{Remark}

\renewenvironment{proof}{\noindent{\bf Proof. }}{\hfill{$\square$} \vskip.3cm}

\newenvironment{notation}{\noindent{\bf Notation. }}{\medskip}

\def\cal{\mathcal}

\def\C{{\mathbb C}}

\def\F{{\mathcal F}}

\def\N{{\mathbb N}}

\def\S{{\mathbb S}}

\def\T{{\mathcal T}}

\def\Z{{\mathbb Z}}

\def\ds{{\displaystyle}}

\def\Rat{{\rm Rat}}
\def\rat{{\rm rat}}

\def\TOF{\bf{\bf{\mathfrak T}}}

\def\Dyn{{{\bf Dyn}}}

 \def\epsilon{{\varepsilon}}

\def\deg{{\rm deg}}
\def\mult{{\rm mult\,}}
\def\Crit{{\rm Crit}}

\usepackage{hyperref}

\hypersetup{
backref=true, 
pagebackref=true,
hyperindex=true, 
colorlinks=true, 
breaklinks=true, 
urlcolor= blue, 
linkcolor= blue, 
}

 \title[Approximability of dynamical systems between trees of spheres]{Approximability of dynamical systems between trees of spheres}                                 
    \author[Matthieu Arfeux]{Matthieu Arfeux}                                
    \address{Pontificia Universidad Católica de Valparaíso, Blanco Viel 596, Cerro Barón, Valparaíso, Chile}                                    
    \email{matthieu.arfeux@pucv.cl}                                      
    \date{}                                       
    \thanks{I would like to thank my advisor Xavier Buff for helping me to make my ideas more clear. This work would not have been possible without a lot of interesting discussions during conferences and collaborations in Santiago de Chile with Jan Kiwi. Also I would like to thank Charles Favre for his efforts teaching me the Berkovich point of view. This paper would not have been the same without the helpful comments form the kind and patient  referee.}                                     
\begin{document}

\begin{abstract}

We study sequences of analytic conjugacy classes of rational maps which diverge in moduli space. In particular, we are interested in the notion of rescaling limits introduced by Jan Kiwi. From \cite{A2},
we recall the notion of dynamical covers between trees of spheres for which a periodic sphere corresponds to a rescaling limit. We study necessary conditions for such a dynamical cover to be the limit of dynamically marked rational maps. With these conditions we classify these covers in the case of bi-critical maps and we  recover the second main result of Jan Kiwi regarding rescaling limits. 

\end{abstract}

    \maketitle


\tableofcontents


\section{Introduction}

\medskip
\noindent{\textbf{Context.}}
Let $\S$ denote the Riemann sphere. We are interested in the space $\Rat_d$ of rational maps of degree $d\geq2$ and its quotient space $\rat_d$ modulo conjugacy by Moebius transformations. More precisely we are interested in the iteration of the elements of $\Rat_d$ on $\S$ and in what happens when we consider a sequence of such elements whose representatives diverge in $\rat_d$. For such sequences, we can have a phenomenon called a rescaling limit as described below.

\begin{definitionint}
For a sequence of rational maps $(f_n)_n$ of a given degree, a rescaling is a sequence of Moebius transformations $(M_n)_n$ such that there exist $k\in\N$ and a rational map $g$ of degree $\geq 2$ such that
$$M_n\circ f_n^k\circ M_n^{-1}\to g$$
uniformly on compact subsets of $\S$ with finitely many points removed.

If this $k$ is minimal then it is called the rescaling period for $(f_n)_n$ at $(M_n)_n$ and $g$ is a rescaling limit for $(f_n)_n$.
\end{definitionint}

This phenomenon was first observed in \cite{S} and inspired Adam Epstein's results in \cite{BHC}. It is related to the existence of indeterminacy points of the iteration map $\Phi_k: \rat_d\to \rat_{d^k}$ defined for $k>1$ by $\Phi_k([f])=[f^k]$. 
In \cite{D1} and \cite{D2}, Laura De Marco studies a compactification of $\rat_d$ giving a continuous extension of $\Phi_k$ through two, a priori different but in fact equivalent, points of view: using the geometric invariant theory or looking at measures of maximal entropy. 
Jan Kiwi was  the first to study rescaling limits explicitly in \cite{Kiwi2}, using diverging families and Berkovich spaces.
After this work, inspired by the Deligne-Mumford compactification of the moduli space of stable curves, in \cite{A},\cite{A2} and \cite{A1},  I introduced another vocabulary to study these phenomena, closer to the standard one used in holomorphic dynamics. 

\medskip
\noindent{\textbf{Trees of spheres and rescaling limits.}}
Denote by $X$, $Y$ and $Z$ three finite sets with at least three elements each and satisfying $X\subseteq Y\cap Z$.
Denote by ${\bf F}$ a {\it portrait} (a map between $Y$ and $Z$ together with a degree function $\deg_{F}:Y\to\N$ satisfying some specific properties) 
and let $\Rat_{{\bf F},X}$ be the set of rational map dynamically marked by $({\bf F},X)$ (the definitions are recalled in Definition \ref{mark} and Definition \ref{dynmark}).
I introduce the natural group of isomorphisms of the elements of $\Rat_{{\bf F},X}$ and define the {\it moduli space of dynamically marked rational map} $\rat_{{\bf F},X}$ to be the quotient of $\Rat_{{\bf F},X}$ under the action of this group by conjugacy.

Then I described a compactification of $\rat_{{\bf F},X}$ by considering it as a subspace of a space of more general rational maps which is compact: the {\it space of dynamical systems between trees of spheres} considered modulo conjugacy by their natural isomorphisms. 

Quite informally, a {\it tree of spheres} $\T^X$ marked by $X$ is a finite tree where the leaves are the elements of a finite set $X$, all internal vertices are spheres, and the edges join marked points or leaves on these spheres (see Figure \ref{exmil} and the formal definition slightly different in Section \ref{etpuiszut}). Still informally, a {\it cover} $\F$ between the trees of spheres $\T^Y$ marked by $Y$ and $\T^Z$ marked by $Z$ is a nice map that assigns to every vertex of $\T^Y$ a vertex of $\T^Z$ and defines a branched cover between them when they are internal vertices (spheres). A {\it dynamical system} between trees of spheres $(\F:\T^Y\to\T^Z,\T^X)$ is a pair consisting of a cover between trees of spheres and a tree of spheres $\T^X$ marked by $X$ with $X\subseteq Y\cap Z$, such that $\T^X$ is {\it compatible with} each of $\T^Y$ and $\T^Z$  (cf example on Figure \ref{cexresc}). The spheres in $\T^X$ are simultaneously spheres of $\T^Y$ and $\T^Z$ so that we can do dynamics. All these definitions will be precisely recalled in Section \ref{Prem}.

In this context, I proved that if a sequence in $\Rat_d$ has a rescaling limit, then, after considering a subsequence and an adequate marking $({\bf F},X)$,  
\begin{itemize}
\item this sequence converges to a dynamical system between trees of spheres $(\F:\T^Y\to\T^Z,\T^X)$, and
\item to each rescaling limit we can associate naturally a unique periodic sphere of $\T^X$ that contains in its cycle a sphere for which the composition of the branched cover along the cycle has degree greater than one (a critical periodic sphere).
\end{itemize}
Conversely we proved that if a dynamical system between trees of spheres is a limit of a sequence of elements in $\rat_{{\bf F},X}$ and has a critical periodic sphere then this sphere can be also naturally associated to a rescaling  limit for this sequence.

On the set of rescaling limits, we recall from \cite{Kiwi2} the two equivalence relations of equivalent rescalings and dynamically dependent rescalings.
We proved that two rescalings are 
\begin{itemize}
\item {\it equivalent} if and only if they can be naturally associated to the same periodic sphere of some limiting dynamical system between trees of spheres and
\item {\it dynamically dependent} if and only if the two naturally associated spheres in some limiting dynamical system between trees of spheres are in the same cycle. 
\end{itemize}
In this paper we will use these characterizations instead of J.Kiwi's explicit definitions.

\medskip
\noindent{\textbf{Objectives.}}
Using the formalism of Berkovich spaces, J.Kiwi proved the two following theorems.

\begin{Thmbis}\label{alpha} \cite{Kiwi2} For every sequence in $\Rat_d$ for $d\geq 2$ there are at most $2d-2$ dynamically independent rescalings classes with a non post-critically finite rescaling limit. 
\end{Thmbis}

\begin{Thmbis}\label{omega} \cite{Kiwi2} Every sequence in $rat_2$ admits at most 2 dynamically independent rescalings limits of period at least $2$. Furthermore, in the case that a rescaling of period at least $2$ exists, then exactly one of the following holds:
\begin{enumerate}
\item $(f_n)_n$ has exactly two dynamically independent rescalings, of periods $q'>q>1$. The period $q$ rescaling limit is a quadratic rational map with a multiple fixed point and a prefixed critical point. The period $q'$ rescaling limit is a quadratic polynomial, modulo conjugacy.
\item $(f_n)_n$ has a rescaling whose corresponding limit is a quadratic rational map with a multiple fixed point and every other rescaling is dynamically dependent to it.
\end{enumerate}
\end{Thmbis}

In \cite{A2} we re-proved Theorem \ref{alpha} after a study of natural properties of general dynamical systems between trees of spheres. Here we re-prove in Theorem \ref{omega}. We will remark that dynamical systems between trees of spheres can have more than two critical periodic cycles (example on Figure \ref{cexresc}). Hence, there exists dynamical system between trees of spheres that cannot be limit of dynamically marked rational maps.
We will describe two necessary conditions for a  dynamical systems between trees of spheres in order to be a limit of dynamically marked rational maps. Those are inspired from Groetsch's Inequality and they will be sufficient to recover Theorem \ref{omega}.

\medskip
\noindent{\textbf{Outline.}}
In Section \ref{Prem} we begin by recalling the main notions introduced in \cite{A2}.  Then we add new ones from subsection \ref{ann}. We define a branch to be a connected component of a tree minus an internal vertex. We define an annulus to be  the non empty intersection of two branches. For $v$ and $v'$ two internal vertices we denote by $[v,v']$ the path between them and by $]\!]v,v'[\![$ the annuli consisting of the intersection of branches on $v$ and $v'$ which has non-empty intersection with $[v,v']$.
 We then prove some properties about branches and annuli of dynamical systems between trees of spheres. 

When a dynamical system of trees of spheres with portrait ${\bf F}$ is the limit of a sequence of dynamical systems of marked spheres diverging in $\Rat_{{\bf F},X}$ we say that it is {\bf approximable} by this sequence. Let $\partial\Rat_{{\bf F},X}$ denote the set of such approximable dynamical systems. We prove in Section \ref{conNecess} two  lemmas (Lemma \ref{branchcrit} and Lemma \ref{annocritiq}) who state necessary conditions for a dynamical system between trees of spheres in order to be approximable. Quite informally, these lemmas say that if an approximable dynamical system between trees of spheres maps an annulus $A$ to another one $A'$ in a position that makes them comparable for the inclusion (in a more precise sense), then
\begin{itemize}
\item either the degree of the cover on $A$ is $1$, then $A=A'$ and we have additional properties relating the dynamics on the border of these annuli, 
\item or this degree is bigger than one and $A\subset A'$.
\end{itemize}

In Section \ref{chap6}, we classify the rescaling limits of the dynamical systems between trees of spheres with exactly two critical leaves (bi-critical).

 \begin{theoremint}[Classification]\label{thmkiw2}\label{class0}
Let ${\bf F}$ be a portrait of degree $d$ with $d+1$ fixed points and exactly $2$ critical points and let $({\cal F},{\cal T}^X)\in\partial\Rat_{{\bf F},X}$. Suppose that there exists $(\displaystyle f_n,y_n,z_n)\overset{\lhd}{\longrightarrow}{ \F}$ in $\Rat_{{\bf F},X}$ such that for every $n$, $x_n(X)$ contains all the fixed points of $f_n$.
 Then the map $\F$ has at most two critical cycles of spheres; they have degree $d$.

Assume that there exits at least one rescaling limit. Then there is a vertex $ w_0$ separating three fixed points which is fixed and such that $f_{ w_0}$ has finite order $k_0>1.$
 Denote by $v_0$ the critical vertex separating $w_0$ and the two critical leaves.
\begin{enumerate} 

\item\label{class2}  Either $v_0$ belongs to a critical cycle of period $k_0$ and
\begin{enumerate} 
	\item\label{class2a} its associated cover has a parabolic fixed point;
	\item\label{class2b} if there is a second critical cycle then it has period $k'_0>k_0$, its associated cover has a critical fixed point with local degree $d$ and the cover associated to $v_0$ has a critical point that eventually maps to the parabolic fixed point.
\end{enumerate}
\item\label{class11}  Or $F^k(v_0)\notin {\cal T^X}$ for some $k<k_0$; in this case there is exactly one critical cycle; it has period $k'_0> k_0$ and its associated cover has a critical fixed point with local degree $d$.
\end{enumerate}
\end{theoremint}


We then deduce a generalization of Theorem \ref{omega} and \cite{A} to the bi-critical case :

\begin{theoremint}\label{omega2}  Every sequence of bi-critical maps in $rat_d$ admits at most 2 dynamically independent rescaling limits of period at least $2$. Furthermore, in the case that a rescaling of period at least $2$ exists, then exactly one of the following holds:
\begin{enumerate}
\item $(f_n)_n$ has exactly two dynamically independent rescalings, of periods $q'>q>1$. The period $q$ rescaling limit is a degree $d$ rational map with a multiple fixed point and a prefixed critical point. The period $q'$ rescaling limit is a degree $d$ polynomial, modulo conjugacy.
\item $(f_n)_n$ has a rescaling whose corresponding limit is a degree $d$ rational map with a multiple fixed point and every other rescaling is dynamically dependent to it.
\end{enumerate}
\end{theoremint}

On Figure \ref{cexresc}, we provide an example of a dynamical system between trees of spheres that is not a limit of dynamically marked rational maps. Indeed, this dynamical system does not satisfy the conclusion of Theorem \ref{class0}.

To conclude, in Section \ref{last} we discuss questions about rescalings on an explicit example and compare our approach with J.Milnor's compactification in the special case of degree $2$.


\section{Preliminaries}\label{Prem}

\subsection{Some definitions and notation}

The content of this subsection already appears in \cite{A2}.

\subsubsection{Covers between trees of spheres}\label{etpuiszut}

 Let $X$ be a finite set with at least 3 elements. 
 A tree of spheres ${\cal T}$ marked by $X$ is the following data : 
\begin{itemize}
\item a combinatorial tree $T$ whose leaves are the elements of $X$ and every internal vertex has at least valence $3$,  
\item for each internal vertex $v$ of $T$, an injection $i_v$ of the set of edges $E_v$ adjacent to $v$ into a projective sphere that we denote by ${\S}_v$.
\end{itemize}

We often use the notation $T$ to emphasize that we talk about the combinatorial tree, whereas we use $\T$ when we use the datas of the spheres.
We denote by $\bf{\mathfrak T}_X$ the set of trees of spheres marked by $X$. 
 Let $X_v := i_v(E_v)$. We let $a_v:X\to {\S}_v$ to be a map given by $a_v(x) := i_v(e)$ if $x$ and $e$ lie in the same connected component of $T-\{v\}$. We denote by $[v,v']$ the path between $v$ and $v'$ including these vertices and by $]v,v'[$ the path $[v,v']$ minus the two vertices $v$ and $v'$.

When the tree has only one internal vertex $v$, this tree is equivalent to the data of the injection $i_v:E_V\to {\S}_v$ and each edge can be identified to the only leaf to which it is adjacent. In this case we get a marked sphere as defined below.

\begin{definition}[Marked sphere]
A sphere marked by $X$ is an injection $$x:X\to \mathbb S.$$
\end{definition}

In this paper we will make an abuse of notation by confounding a tree marked by $X$ with a unique internal vertex and  
the corresponding marked sphere. 

Let us recall the notion of rational maps marked by a portrait:

\begin{definition}[Marked rational maps]\label{mark}
A rational map marked by ${\bf F}$ is a triple $(f,y,z)$ where
\begin{itemize}
\item $f\in \Rat_d$
\item $y:Y\to \S$ and $z:Z\to \S$ are marked spheres, 
\item $f\circ y = z\circ F$ on $Y$ and 
\item $\deg_{y(a)}f = \deg(a)$ for $a\in Y$. 
\end{itemize}
\end{definition}

Where a portrait ${\bf F}$ of degree $d\geq 2$ is a pair $(F,\deg)$ such that 
\begin{itemize}
\item $F:Y\to Z$ is a map between two finite sets $Y$ and $Z$ and
\item $\deg:Y\to \N-\{0\}$ is a function that satisfies
\[\sum_{a\in Y}\bigl(\deg(a) -1\bigr) = 2d-2\quad\text{and}\quad \sum_{a\in F^{-1}(b)} \deg(a) = d\quad\text{ for all } b\in Z.\] 
\end{itemize}

If $(f,y,z)$ is marked by ${\bf F}$, we have the following commutative diagram : 

\centerline{
$\xymatrix{
 Y \ar[r]^{y}    \ar[d] _{{ F}} &\S  \ar[d]^{f} \\
      Z \ar[r]_{z}  &\S
  }$}

Typically, $Z\subset \S$ is a finite set, $F:Y\to Z$ is the restriction of a rational map $F:\S\to \S$ to $Y:=F^{-1}(Z)$ and $\deg(a)$ is the local degree of $F$ at $a$. In this case, the Riemann-Hurwitz formula and the conditions on the function $\deg$ imply that $Z$  contains the set $V_F$ of the critical values of $F$ so that $F:\S-Y\to \S-Z$ is a covering map.

As trees of spheres can be understood as a generalization of marked spheres, we define a generalization of marked rational maps as follows. 
A (holomorphic) cover ${\cal F}:{\cal T}^Y\to {\cal T}^Z$ between $Y\in\TOF_Y$ and $Z\in\TOF_Z$ is the following data
\begin{itemize}
\item a map $F:T^Y\to T^Z$ mapping leaves to leaves, internal vertices to internal vertices, and edges to edges, 
\item for each internal vertex $v$ of $T^Y$ and $w:=F(v)$ of $T^Z$, a holomorphic ramified cover $f_v:{\S}_v\to {\S}_w$ that satisfies the following properties: 
\begin{itemize}
\item the restriction $f_v : {\S}_v-Y_v\to {\S}_w-Z_w$ is a cover, 
\item $f_v\circ i_v = i_w\circ F$,
\item if $e$ is an edge between $v$ and $v'$, then the local degree of $f_v$ at $i_v(e)$ is the same as the local degree of $f_{v'}$ at $i_{v'}(e)$. 
\end{itemize}
\end{itemize}
In \cite{A2} we proved that a cover between trees of spheres ${\cal F}$ is surjective and has a global degree that we denote by $\deg({\cal F})$.

\subsubsection{Dynamical systems between trees of spheres}

Suppose in addition that $X\subseteq Y\cap Z$. Under some assumptions, we can associate a dynamical system to covers between trees of spheres. More precisely we will say that $({\cal F},{\cal T}^X)$ is a dynamical system of trees of spheres if :
\begin{itemize}
\item ${\cal F}:{\cal T}^Y\to {\cal T}^Z$ is a cover between trees of spheres, 
\item ${\cal T}^X$ is a tree of spheres compatible with ${\cal T}^Y$ and ${\cal T}^Z$, ie : 
\begin{itemize}
\item $X\subseteq Y\cap Z$
\item each internal vertex $v$ of $T^X$ is an internal vertex common to $T^Y$ and $T^Z$, 
\item  ${\S}_v^X = {\S}_v^Y = {\S}_v^Z$ and 
\item $a_v^X = a_v^Y|_X = a_v^Z|_X$. 
\end{itemize}
\end{itemize}
We denote by $\Dyn_{{\bf F},X}$ the set of dynamical system of trees of spheres defined this way.
We have the very useful following lemma:
\begin{lemma} \label{definiX}
If $T^X$ is compatible with $T^Y$ and if an internal vertex $v$ of $T^Y$ separates three vertices of $T^X$, then $v\in T^X$. 
\end{lemma}

With this definition we are able to compose covers along an orbit of vertices as soon as they are in $T^X$. When it is well defined we will denote by $f_v^k$ the composition $f_{F^{k-1}(v)}\circ\ldots\circ f_{F(v)}\circ f_v$.

Dynamical covers between marked spheres can be naturally identified to dynamically marked rational maps:

\begin{definition}[Dynamically marked rational map]\label{dynmark}
A rational map dynamically marked by  $({\bf F},X)$ is a rational map $(f,y,z)$ marked by  ${\bf F}$ such that  $y|_X=z|_X$. 
\end{definition}

We denote by 
 $\Rat_{{\bf F},X}$ the set of rational maps dynamically marked by $({\bf F},X)$. We identify $\Rat_{{\bf F},X}$ to the subset of $ \Dyn_{{\bf F},X}$ consisting of dynamical system between the trees of spheres $\T^Y$ and $\T^Z$ where $\T^Y\in\TOF_Y$ and $\T^Z\in\TOF_Z$, ie when $\T^Y$ and $\T^Z$ are identified with marked spheres.

Let $(\F,{\T}^X)\in \Dyn_{{\bf F},X}$. A period $p\geq 1$ cycle of spheres is a collection of spheres $({\S}_{v_k})_{k\in \Z/p\Z}$ where the $v_k$ are internal vertices of $T^X$ that satisfies $F(v_k)=v_{k+1}$. The cycle is critical if it contains a critical sphere, ie a sphere $\S_v$ such that $\deg(f_v)$ is greater or equal to two. If a sphere $\S_v$ on a critical cycle contains a critical point of its respective $f_v$ that has infinite orbit, then the cycle is said to be non post-critically finite.

\subsubsection{Convergence notions and approximability}
The following results are proven in \cite{A2}.

 \begin{definition}[Convergence of marked spheres]
A sequence ${\cal A}_n$
 of marked spheres $a_n:X\to{\mathbb S}_n$
  converges to ${\cal T}^X\in\TOF_X$
  if for every internal vertices $v$
 of $\T^X$ , there exists an isomorphism $\phi_{n,v}:{\mathbb S}_n\to{\S}_v$
    such that $\phi_{n,v} \circ a_n$
     converges to $a_v$. 
     \end{definition}
We write ${\cal A}_n\to {\cal T}^X$ or $\displaystyle {\cal A}_n\underset{\phi_n}\longrightarrow {\cal T}^X$. 
    
The following lemma explains in which sense the $\phi_{n,v}$ above depends on the sphere $\S_v$.

\begin{lemma} \label{noncomp}
Let $v$ and $v'$  be two distinct internal vertices of $\T^X$ and $(\cal A_n)_n$ be a sequence of marked spheres such that $\cal A_n\underset{\phi_n}\longrightarrow \T^X$. 
Then the sequence of isomorphisms $(\phi_{n,v'} \circ \phi_{n,v}^{-1})_n$
converges locally uniformly outside $i_v(v')$ to the constant $ i_{v'}(v)$. 
\end{lemma}

\begin{definition}[Non dynamical convergence]
Let ${ \F}:{\T}^Y\to { \T}^Z$ be a cover between trees of spheres of portrait ${\bf F}$. A sequence ${ \F}_n:=(f_n,a_n^Y,a_n^Z)$ of marked spheres covers converges to ${ \F}$ if their portrait is ${\bf F}$ and if for every pair of internal vertices $v$ and $w:=F(v)$, there exist sequences of isomorphisms $\phi_{n,v}^Y:\S_n^Y\to \S_v$ and $\phi_{n,w}^Z:\S_n^Z\to \S_w$ such that 
\begin{itemize}
\item $\phi_{n,v}^Y\circ a_n^Y:Y\to \S_v$ converges to $a_v^Y:Y\to \S_v$, 
\item $\phi_{n,w}^Z\circ a_n^Z:Z\to \S_w$ converges to $a_w^Z:Z\to \S_w$ and 
\item $\phi_{n,w}^Z\circ f_n\circ (\phi_{n,v}^Y )^{-1}:\S_v\to \S_w$ converges locally uniformly outside $Y_v$ to ${f_v:\S_v\to \S_w}$. 
\end{itemize}
\end{definition}

We write $\F_n\rightarrow \F$ or $\F_n\underset{(\phi^Y_n,\phi^Z_n)}\longrightarrow  \F.$

\begin{definition}[Dynamical convergence]\label{defcvdyn}
 A sequence $({ \F}_n)_n$ in $\Rat_{{\bf F},X}$ converges to $({\F},{ \T}^X)\in\Dyn_{{\bf F},X}$ 
if  $$\displaystyle { \F}_n\underset{\phi_n^Y,\phi_n^Z}\longrightarrow{ \F}\quad\text{with}\quad\phi_{n,v}^Y=\phi_{n,v}^Z$$ for every vertices $v\in IV^X$. We say that $({ \F},{ \T}^X)$ is approximable by $({ \F}_n)_n$.
\end{definition}

We write  $$\displaystyle { \F}_n\overset{\lhd}{\underset{\phi_n^Y,\phi_n^Z}\longrightarrow}{ \F}.$$ We denote by $\partial \Rat_{{\bf F},X}$ the subset of $\Dyn_{{\bf F},X}$ consisting of elements approximable by a sequence in $\Rat_{{\bf F},X}$.

The following lemma and its corollary explain relate the convergence in the setting of trees of spheres to the convergence in $\rat_D$.

\begin{lemma}\label{cvu}
Let ${\cal F}:{\cal T}^Y\to{\cal T}^Z$ be a cover between trees of spheres with portrait ${\bf F}$ and degree $D$. Let
$v\in IV^Y$ with $\deg(v)=D$ and let ${\cal
F}_n:=(f_n,a_n^Y,a_n^Z)$ be a sequence of covers between trees of spheres that satisfies $\displaystyle {\cal F}_n\underset{\phi^Y_n,\phi^Z_n}\longrightarrow{\cal F}$. Then the sequence $\phi_{n,F(v)}^Z\circ f_n\circ (\phi_{n,v}^Y)^{-1}:\S_v\to \S_{F(v)}$ converges uniformly
to $f_v:\S_v\to \S_{F(v)}$.
\end{lemma}

\begin{corollary}\label{cvuutil}
Let $({\cal F},{\cal T}^X)\in\Dyn_{{\bf F},X}$ dynamically approximable by ${\cal F}_n:=(f_n,a_n^Y,a_n^Z)$. Suppose that $v\in IV^X$ is a fixed vertex such that $\deg(v) =D= \deg({\bf F})$. Then the sequence $[f_n]\in \rat_D$ converges to the conjugacy class $[f_v]\in \rat_D$. 
\end{corollary}

The following lemma explains the relation between the combinatorial map $F$ and the $\phi_{n,v}$. For $k\in\N$, we denote by $IV^{F^k}$ the set of internal vertices of $T^Y$ on which $F^k$ is well defined.

\begin{lemma}\label{convit}
Let $({ \F},{ \T}^X)$ be dynamically approximable by $({ \F}_n)_n$. If $v\in IV(F^k)$ and $w:=F^k(v)$, then $(\phi_{n,w}\circ f_n^k\circ \phi_{n,v}^{-1})_n$ converges locally uniformly to $f_v^k$ outside a finite set.
\end{lemma}

We can deduce the following corollary. Quite informally, this corollary explains that if we use $w\neq F^{k}(v)$ in the formula above, then $\phi_{n,w}\circ f_n^k\circ \phi_{n,v}^{-1}$ converges to a constant pointing the the branch on $\S_w$ that contains the $F^{k}(v)$.

\begin{corollary}\label{divrev} Let $({\cal F:\T^Y\to\T^Z},{\cal T}^X)$ be dynamically approximable by $(\F_n)_n$. Let ${v\in IV(F^k)}$, let $w\neq v$ be another internal vertex of $\T^Z$ and let $e\in E_w$ be such that $F^k(v)\in B_w(e)$. Then $(\phi_{n,w}\circ f_n^k\circ \phi_{n,v}^{-1})_n$ converges locally uniformly to the constant $i_w(e)$ outside a finite set.
 \end{corollary}

\begin{proof}
Indeed, we have $$\phi_{n,w}\circ f_n^k\circ \phi_{n,v}^{-1}=(\phi_{n,w}\circ\phi^{-1}_{n,F^k(v)})\circ(\phi_{n,F^k(v)}\circ f_n^k\circ \phi_{n,v}^{-1}).$$
According to lemma  \ref{convit}, the map on the right converges locally uniformly outside a finite set to a non constant map. According to lemma  \ref{noncomp}, the map on the left converges locally uniformly to $i_w(e)$ outside $i_{F^k(v)}(w)$. The conclusion follows. 
\end{proof}

\subsection{Branches and annuli}\label{ann}

Given a tree $T$, the most natural open subgraphs to look at are those defined in the following.
\begin{definition}[Branch]
For $v$ a vertex of a tree $T$ and for $\star\in T-\{v\}$, a branch of $\star$ on $v$ is the connected component of $T-\{v\}$ containing $\star$. It is denoted by $B_v(\star)$.
\end{definition} 

\begin{definition}[Annulus]\label{defann}
If $v_1$ and $v_2$ are two distinct internal vertices of $T$, the annulus $A:=]\!]v_1,v_2[\![$ is the intersection of two branches $B_{v_1}(v_2)$ and $B_{v_2}(v_1)$. We define $$[\![v_1,v_2]\!]:=\overline A:=A\cup[v_1,v_2].$$
\end{definition}

Note that $\overline A=A\cup\{v_1,v_2\}$. More generally, for every connected subset $T'$ of the tree $T$, we denote by $\overline T'$ the smallest subtree of $T$ containing $T'$. We proved in \cite{A2} that given $T''$ an open, non empty and connected subset of $T^Z$ and given $T'$ a connected component of $F^{-1}(T'')$ there is a natural cover $\overline {\cal F}:\overline {\cal T}'\to \overline {\cal T}''$ defined by 
\begin{itemize}
\item $\overline F:= F: \overline T'\to \overline T''$ and 
\item $\overline f_v:= f_v$ if $v\in V'-Y'$
\end{itemize}
which is a cover between trees of spheres. 
Recall the Riemann-Hurwitz formula for trees proven in \cite{A2}:

\begin{proposition}[Riemann-Hurwitz Formula]\label{RH} Let ${\cal F}:{\cal T}^Y\to {\cal T}^Z$ be a cover between trees of spheres and $T''$ be a sub-tree of $T^Z$. Let $T'$ be a connected component of $F^{-1}(T'')$. Then we have $$\chi_{T^Y}(T')=\deg(\F|_{\overline{\cal T'}})\cdot \chi_{T^Z}(T'')-\sum_{y\in \Crit\F\cap T'}\mult(y).$$
\end{proposition}

From this formula we deduce the following useful corollaries.


\begin{corollary}\label{branchdeg1}
If $B$ is a branch of $T^Y$ that does not contain any critical leaf, then its image $F(B)$ is a branch and $F:B\to F(B)$ is a bijection. 
\end{corollary}

\begin{proof}
Set $v\in V$ and $e$ such that $B=B_v(e)$. Let $T''$ be the branch of $F(e)$ containing $F(v)$ and $T'$ be the component of $F^{-1}(T'')$ containing $v$. Then, $T'$ is a sub-tree of $B$ and it follows that $\mult T'=0$. From the Riemann-Hurwitz formula,
\[ \chi_{T^Y}(T') = {\rm deg}(F:T'\to T'')\cdot \chi_{T^Z}(T'') - \mult T' = {\rm deg}(F:T'\to T'')\cdot \chi_{T^Z}(T'').\] 
It follows that ${\rm deg}(F:T'\to T'')=1$ and $\chi_{T^Y}(T')=1$. In particular, $T'$ is a branch so $B=T'$. Moreover, $F(B)=T''$ and the degree of $F:B\to F(B)$ is equal to $1$.
\end{proof} 

Using the same ideas we can find a formula for annuli.

\begin{corollary} \label{anneaudeg1} 
If $A$ is an annulus of $T^Y$ that does not contain a critical leaf, then its image $F(A)$ is an annulus and $A$ is a connected component of $F^{-1}(F(A))$. 

Moreover we have $F(\overline A)=\overline{F(A)}$. If in addition $A=]\!]v_1,v_2[\![$ does not contain any critical element, then $F:A\to F(A)$ and $\overline F:\overline A\to \overline{F(A)}$ are bijections.  
\end{corollary}

\begin{proof} Recall that for two adjacent vertices $v, v'$ of a graph, the edge between is by definition the set $\{v,v'\}$ (see \cite{A2} for notations for graphs).
Let $A=]\!]v_1,v_2[\![$.
If $A=\{\{v_1,v_2\}\}$ then the result follows directly from the definition of combinatorial trees maps. Suppose that it is not the case. Recall that $\overline A$ is the sub-tree of $T^Y$ defined by adding $v_1$ and $v_2$ to the set of vertices of $A$.

Let $e_1$ and $e_2$ be two edges connecting $v_1$ and $v_2$ to the rest of $\overline A$,  $v$ a vertex of $\overline A$. Let $T''$ be the connected component of the graph $(V^Z, E^Z-\{F(e_1),F(e_2)\})$ containing $F(v)$ and $T'$ be the component of $F^{-1}(T'')$ containing $v$. Then $T'$ is a sub-tree of $\overline A$ and $\mult (T')=0$. From the Riemann-Hurwitz formula 
\[0=\chi_{T^Y}(\overline A)\geq \chi_{T^Y}(T') = {\rm deg}(\F:\T'\to \T'')\cdot \chi_{T^Z}(T'').\]
The connected components of the graph $(V^Z, E^Z-\{F(e_1),F(e_2)\})$ have characteristic positive or equal to zero in $T^Z$. Then, 
\[0=\chi_{T^Y}(A')=\chi_{T^Y}(T') =\chi_{T^Z}(T'').\]
This proves that $\overline A=T'$ and $F(\overline A)=T''$. 

Now suppose that $A$ does not contain any critical vertex. Then the edges $e_1$ and $e_2$ have degree one and the map $\overline {\cal F}:\overline {\cal T}'\to \overline{\cal T}''$ has no critical leaves. So it has degree one. This proves that $F:\overline A\to F(\overline A)$ is a bijection. In particular $F:A\to F(A)$ is a bijection.

Given that images of two adjacent vertices are adjacent vertices, the results about the closure follow.
\end{proof}

\begin{corollary}\label{attachbranch}
Let $B$ be a branch on $v$ in $T^Y$ containing at most one critical leaf $c$. Then $F(B)$ is the branch on $F(v)$ attached at $ a_{F(v)}(F(c))$. 
\end{corollary}

\begin{proof}
According to Lemma \ref{anneaudeg1}, $F([\![ v,c]\!])=[\![ F(v),F(c)]\!]$ so in particular $F(B)=F(]\!] F(v),F(c)[\![)\cup \{ F(c)\}$ which is a branch on $F(v)$. Otherwise the edge of $B$ on $v$ maps to the edge attached at $ a_{F(v)}(F(c))$. 
\end{proof}

We remark that Corollary  \ref{branchdeg1} and Corollary \ref{anneaudeg1} can also be proved from the following useful lemma:

\begin{lemma}\label{chemdeg1} 
If $[v_1,v_2]$ is a path in $T^Y$ having only vertices of degree one, then $F$ is a bijection from $[v_1,v_2]$ to $[F(v_1),F(v_2)]$.
\end{lemma}

\begin{proof} Let $T':=[v_1,v_2]$. Consider a leaf of $F(T')$. It is the image of some vertex $v\in T^Y$. Let us prove that $v=v_1$ or $v=v_2$.
Suppose that $v$ is not a leaf of $T'$. Then the two edges $e$ and $e'$ of $v$ in $T'$ map to the unique edge of $F(v)\in F(T')$. Then $i_v(e)$ and $i_v(e')$ have same image on $\S_{F(v)}$, which contradicts the fact that $f_{v}$ has degree one.
It follows that the only leaves of $F(T')$ are $F(v_1)$ and $F(v_2)$. As $F(T')$ is connected, we deduce that $F(T')=[v_1,v_2]$.
\end{proof}


\section{Necessary conditions for approximability}\label{conNecess}

In this section we look at the properties of the elements of $\partial \Rat_{{\bf F},X}$.

\subsection{Branches lemma}

In this section, we prove the following result :

\begin{lemma}[Branches]\label{branchcrit} Let $(\F,\T^X)\in\partial\Rat_{{\bf F},X}$,  $v$ a periodic internal vertex, $t_0\in \S_v$ and let $B$ be a branch on $v$ such that for every $k\in \N$, the branch of $T^Y$ attached to $f^k(t_0)$ maps inside the branch on $T^Z$ attached to $f^{k+1}(a_0)$. 
Then 
\begin{itemize}
\item $B$ does not contain a critical periodic vertex ;
\item if $B$ contains a periodic internal vertex then its cycle has degree $1$ and $t_0$ is periodic. 
\end{itemize}
\end{lemma}

 We first remark that the Riemann-Hurwitz formula provides the following result:

\begin{lemma}\label{corollaribilisation}
If a branch $B$ on a vertex $v$ maps to a branch, and if $d$ is the degree of the attaching point of the edge of $B$ on $v$, then the number of critical leaves in $B$, counting multiplicities, is $d-1$. 
\end{lemma}

We are now ready to prove the branches lemma.

\begin{proof}[Lemma \ref{branchcrit}]
For $\star\in\{X,Y,Z\}$, we denote by $B_k^\star$ the branch attached to $f^k(a_0)$ in $T^\star$. Let $v'$ be a periodic vertex in $B^X_0$. Let $v_k:=F^k(v)$ and $v'_k:=F^k(v')$.
As the iterates of $v'$ are in the $B_k^\star$, the orbit of $a_0$ is periodic under $f$.

Suppose that $t_0$ is periodic with period $k_0$ and that the vertex $v'$ has exact period $k'$ which is a multiple of $k_0$. 
Then $v_0$ and $v'_0$ lie in $Z$, we find $z$ and $z'\in Z$ such that the path $[z,z']$ in $T^Z$ passes through $v_0$ and $v'_0$ in this order. 
 For $\star\in\{v_0,v'_0\}\subset T^Z$, we give projective charts $\sigma_\star$ such that
\[\sigma_\star\circ a_\star(z) = \infty\quad\text{and}\quad  \sigma_\star\circ a_\star(z') = 0.\]

Suppose that the sequence ${ \F}_n:=(f_n,a_n^Y,a_n^Z)$ of marked spheres covers converges to ${ \F}$ and take $\phi_n^Y$ and $\phi_n^Z$ like in Definition \ref{defcvdyn}. After post-composing the isomorphisms $\phi_{n,\star}^Z$ by automorphisms of $\S_\star$ tending to identity when $n\to \infty$, we can suppose that 
\[\phi_{n,\star}^Z\circ a_n(z) = a_\star(z)\quad\text{and}\quad \phi_{n,\star}^Z\circ a_n(z') = a_\star(z').\]
Define the projective charts on $\S_n$ by $\sigma_{n,\star}:=\sigma_\star\circ \phi_{n,\star}^Z$. 

For every $n\in\N$, the change of coordinates $\sigma_{n,v'}\circ \sigma_{n,v}^{-1}$ fixes $0$ and $\infty$. So it is a dilatation centered at $0$, i.e. we can define $\lambda_n\in\C$ such that $\sigma_{n,v'}=\lambda_n \sigma_{n,v}$. 
The vertices $v$ and $v'$ have at least three edges, so Lemma \ref{noncomp} ensures that $\lambda_n\to \infty.$

Let $D\subset \S_v$ be a disk containing $a_v(z')$. 
Given that the $f_v$ are continuous for all the $v\in\T^Y$, we can suppose that $D$ is sufficiently small such that the set of its $k'$ iterates under $f$ contains at most a unique attaching point which is the iterate of $t_0$.

Let $D_n:=(\phi^{Z}_{n,v})^{-1}(D)$. Now we show that in the chart $\sigma_{n,v}$, we can take $n$ large enough such that the map $f^{k'}_n$ has no poles on the disk $D_n$. We have $F(B_{k-1}^Y)\subseteq B^Z_{k}$ for every $k>0$, so for $n$ large enough the disk $D_{n,k}$ doesn't intersect attaching points of edges of elements $z\in Z-B^Z_{k}$.
We have $\phi_{n,v}^Z(\partial D_{n,k})\to\partial D_k$, so from the maximum modulus principle we conclude that $ \phi_{n,v}^Z(D_{n,k})\to D_k$. By definition of $D$, $D_{0}=D$ does not contain poles of $f^{k'}$ so $D_n$ does not contain poles of $f^{k'}_n$. 

In the charts $\sigma_{n,v}$ and $\sigma_{n,v'}$, we can develop the map $f_n^{k'}:D_n\to \S_n-\{a_{n}(z)\}$ as a power series in $a_n(z')$ :
\[\sigma_{n,v}\circ f_n^{k'}=\sum_{j\in \N} c_{n,j}\cdot \sigma_{n,v}^j \quad {\rm and}\quad\sigma_{n,v'}\circ f_n^{k'}=\sum_{j\in \N} c'_{n,j}\cdot{\sigma^j_{n,v'}}.\]

We have 
$\sigma_{n,v'} = \lambda_n \sigma_{n,v}$,
so
\[c'_{n,j} = \lambda_n^{i-j} c_{n,j}.\]

Considering the Laurent series of ${f^{ {k'}}:\S_v\to \S_v}$ in the neighborhood of $t_0=a_{v}(z')$ in the chart $\sigma_{v}$ and that of $f^{ {k'}}:\S_{v'}\to \S_{v'}$ in the neighborhood of $a_{v'}(z')$ in the chart $\sigma_{v'}$,
\[\sigma_{v} \circ f^{k'}= \sum_{j\in\Z} c_j\cdot\sigma_{v}^j\quad\text{ and }\quad\sigma_{v'}\circ f^{k'} = \sum_{j\in\Z} c'_j\cdot\sigma_{v'}^j.\]

If $F^{k}(v)=v$ then Lemma \ref{convit} ensures that $\phi_{n,v} \circ f_n^k \circ \phi_{n,v}^{-1}$ converges locally uniformly to $\sigma_v\circ f^k\circ \sigma_v^{-1}$
on $D-\{a_v(z')\}$ so uniformly on $D$ by the maximum principle, since these maps have no poles in $D$. In the chart $\sigma_v$, we then have the convergence $\ds c_{n,j}\underset{n\to \infty}\longrightarrow  c_j$. 
If $F^k(v)\neq v$, Lemma \ref{divrev} allows us to conclude that $\ds c_{n,j}\underset{n\to \infty}\longrightarrow  0=c_j$.
Likewise, we have the convergence $\ds c'_{n,j}\underset{n\to \infty}\longrightarrow  c'_j$. In particular, as $c'_{n,j} = 0$ for $j<0$, we have $c'_j=0$ and $c_j=0$ (we recovered the fact that $f^{k'}$ fixes $t_0$).

If $F^k(v) \neq v$,  all the coefficients $c_j$ are zero, so coefficients of  $c'_j$ are zero, this is a contradiction. 
We conclude that $F^k(v)=v$.

 If we denote by $d$ the local degree of $f^{k'}$ at $t_0$, then $c_j=0$ for $j<d$ and $c_d\neq 0$. If $n\to \infty$ on $c'_{n,j} = \lambda_n^{1-j} c_{n,j}$, then $\lambda_n\to \infty$, and we have $c'_1=c_1$ and $c'_j= 0$ if $j>1$. thus, if $d>1$, the coefficients $c'_j$ are again zero, another contradiction. 
So the only possible case is $d=1$. 
But $d$ is the product of the $d_k$, where $d_k$ is the degree of the attaching point of the edge of $B_k^Y$. Thus, if $d=1$, then all the $d_k$ are 1, then all the branches $B_k^Y$ don't contain critical vertices according to Lemma \ref{corollaribilisation} applied with $d=1$. Thus $f^{k'}$ has degree 1, then the map $f^{k'}:\S_{v'}\to \S_{v'}$ has degree 1 and this is absurd. 
\end{proof}


\subsection{Lemmas about annuli}

In this part we continue the study of covers between trees of spheres which are approximable by a sequence of dynamical systems of marked spheres. 
Recall that according to Corollary \ref{anneaudeg1}, an annulus that does not contain critical leaves has a well defined degree.
We prove the following lemma.

\begin{lemma}[Annuli]\label{annocritiq}\label{multiplic} 
Suppose that 
 $({ \F},{ \T}^X)\in\partial\Rat_{{\bf F},X}$ 
and that $v$ and $v'$ are distinct internal vertices of $\T^Y$ such that for $0\leq k\leq k_0-1$
 the annulus $]\!]F^k(v),F^k(v')[\![^Y$ is defined and does not contain any critical leaf.

\begin{itemize}
\item(Critical) If $\F$ has degree more than $1$ on one of these $]\!]F^k(v),F^k(v')[\![^Y$, we never have $[v_{k_0},v'_{k_0}]\subseteq [v,v']$.

\item(Non critical) If this is not the case and if $[v_{k_0},v'_{k_0}]\subseteq [v,v'] \text{ or }[v,v']\subseteq [v_{k_0},v'_{k_0}]$ we have 
\begin{enumerate}[label=\Alph*)]
	 	\item either $v_{k_0}=v$ and $v'_{k_0} = v'$, then $i_v(v')$ and $i_{v'}(v)$ are fixed by $f^{k_0}$ and the product of the associated multipliers is 1;
  		\item or $v_{k_0}=v'$ and $v'_{k_0} = v$, then $f^{k_0}$ exchanges $i_v(v')$ and $i_{v'}(v)$ and the multiplier of the associated cycle is 1.
 \end{enumerate}
\end{itemize}
\end{lemma}

Note that Figure \ref{cexresc} shows an example of a dynamical system between trees of spheres that does not satisfy the conclusion of the critical annulus lemma (\ref{annocritiq}): the critical annulus between the full red and full black vertices has period 8.

Before proving this result we must note an interesting and open question:

\begin{question} The Annuli Lemma together with the Branches Lemma give some necessary conditions in order to be approximable by a sequence of dynamical systems of spheres. But are they sufficient conditions?
\end{question}

 \begin{figure}
  \centerline{\includegraphics[width=12cm]{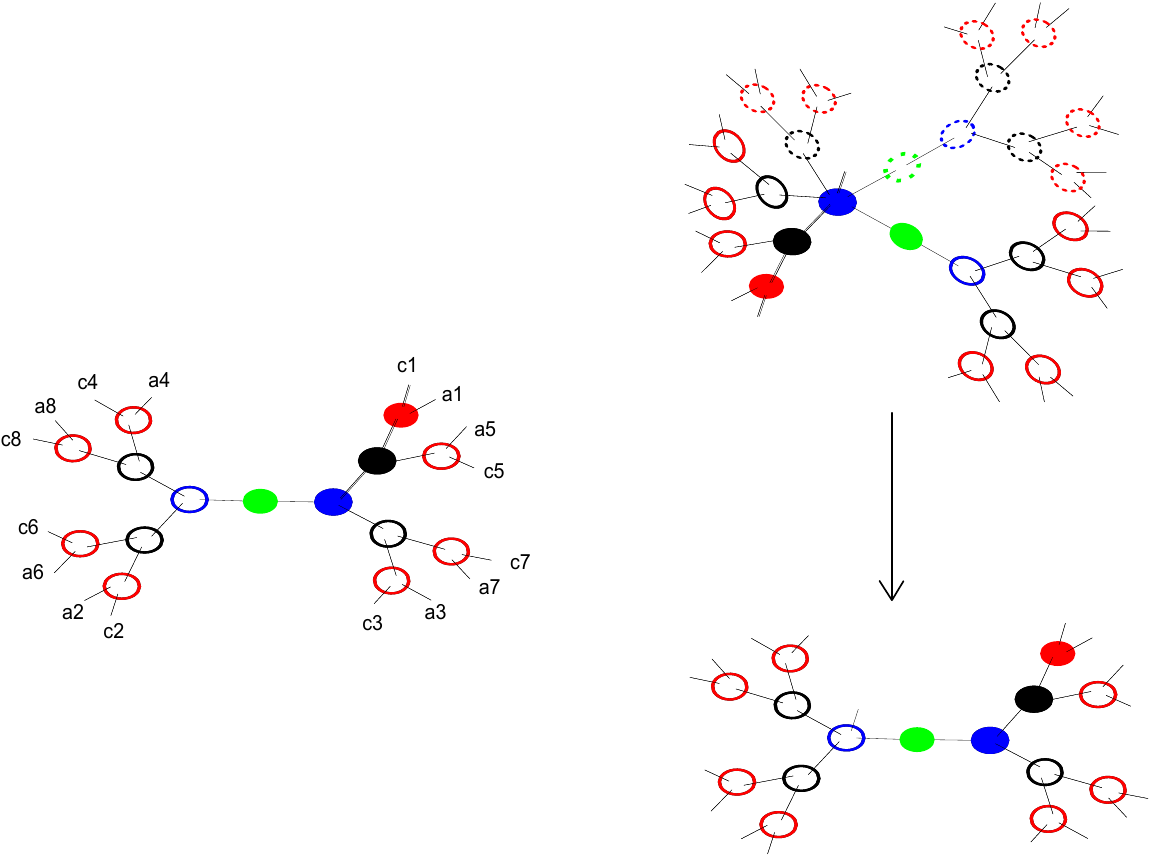}}
   \caption{An example of dynamical system of degree 2 that does not satisfy the conclusion of Theorem \ref{class0}. To the left $\T^X$, to the right top $\T^Y$ and below $\T^Z$.
   The leaves $c_i$ the period 8 cycle of the critical point. The vertices $a_i$ are a cycle of same period. On $\T^X$, the full red (resp. full black, full blue) vertex is critical and of period 8 (resp. 4, 2) and its orbit is the set of red (resp. black, blue) vertices. The full green vertex is fixed.
   }
\label{cexresc} \end{figure}

The rest of this section is dedicated to the proof of this lemma.

We suppose that $({ \F}:{ \T}^Y\to { \T}^Z,{ \T}^X)$  is a tree of spheres  dynamical system dynamically approximable by the sequence of dynamical systems between marked spheres $(f_n,a^Y_n,a^Z_n)$.

\medskip
\noindent{\textbf{Choice of coordinates.}}   
Let $w:=v_{k_0}$ and $w':=v'_{k_0}$. Suppose that ${[w,w']^Z\subseteq[v,v']^Z}$ or that $[v,v']^Z\subseteq [w,w']^Z$. In the non critical case, we have to prove that ${[w,w']^Z=[v,v']^Z}$ and in the critical case that ${[w,w']^Z\not\subseteq [v,v']^Z}$. 

Let $z,z'\in Z$ be such that the path $[z,z']^Z$ goes through $v$ and $v'$ in this order and through $w$ and $w'$, not necessarily in this order. For every internal vertex $\star$ on the path $[z,z']^Z$, we give projective charts $\sigma_\star$ such that
\[\sigma_\star\circ a_\star(z) = 0\quad\text{and}\quad  \sigma_\star\circ a_\star(z') = \infty.\]

Take ${ \F}_n:=(f_n,a_n^Y,a_n^Z)$, $\phi_n^Y$ and $\phi_n^Z$ as in Definition  \ref{defcvdyn}. Let $\S_n$ be such that $a^Y_n:Y\to\S_n$ and $a^Z_n:Z\to\S_n$. Then, after post-composing the isomorphisms $\phi_{n,\star}^Z$ by automorphisms of $\S_\star$ tending to the identity when $n\to \infty$, we can suppose that for every internal vertex $\star$ on the path $[z,z']$ we have
\[\phi_{n,\star}^Z\circ a_n(z) = a_\star(z)\quad\text{and}\quad \phi_{n,\star}^Z\circ a_n(z') = a_\star(z').\]
We then define projective charts on $\S_n$ by $\sigma_{n,\star}:=\sigma_\star\circ \phi_{n,\star}^Z$. 

The changes of coordinate maps fixe $0$ and $\infty$ so they are dilatations centered on $0$, hence we define $\lambda_n$, $\mu_n$, $\rho_n$ and $\rho'_n$ by (cf Figure \ref{bigdrawing}): 
\begin{align*}\sigma_{n,v'}=\lambda_n \sigma_{n,v}&\text{ and }
\sigma_{n,w'}=\mu_n \sigma_{n,w}\\
\sigma_{n,v}=\rho_n \sigma_{n,w}&\text{ and }
\sigma_{n,v'}=\rho'_n \sigma_{n,w'}.
\end{align*}

Note that as the vertices $v,w,v'$ and $w'$ have at least three edges, Lemma \ref{noncomp} implies that the behavior of $\lambda_n, \mu_n,\rho_n$ and $\rho'_n$ is imposed by the relative positions of the vertices (cf below in section Conclusion and multipliers).

\begin{figure}[htbp]
\begin{minipage}[c]{.45\linewidth}
\begin{center}

\includegraphics[width=6.5cm]{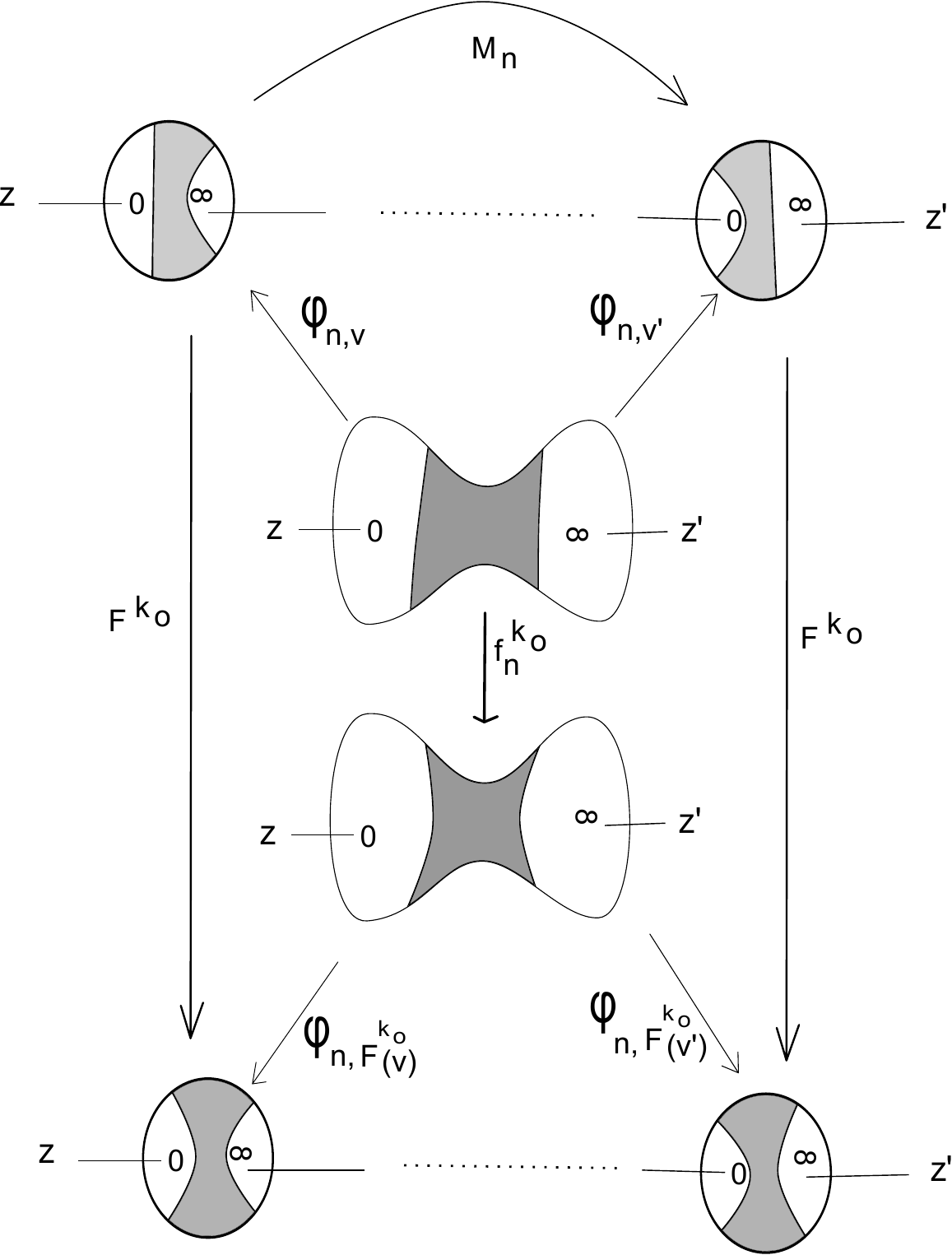}

\end{center}
\end{minipage}
\hfill
\begin{minipage}[c]{.45\linewidth}
\begin{center}

\includegraphics[width=3.5cm]{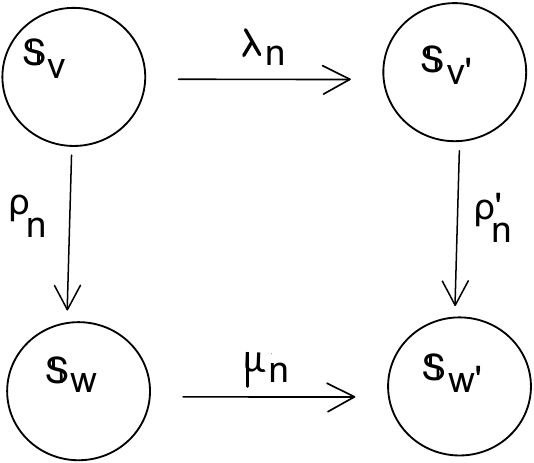}
\label{drawing7} 

\end{center}
\end{minipage}
\caption{A simplified representation of the notation in the proof of Lemma \ref{multiplic}.}\label{bigdrawing}
\end{figure}

\noindent{\textbf{Behavior of annuli.}}  (cf Figure \ref{bigdrawing}) 
Let $Y(f^{k_0})$ be the set of points on the spheres of $\T^Y$ for which the image by an iterate of $f_v^k$ (for some $v\in \T^Y$) with $k\in [1,k_0]$ is the attaching point of an edge in ${ \T}^Z$. This is a finite set containing $i_v(v')$ and $i_{v'}(v)$. 

Let $D\subset \S_v$ (respectively $D'\subset \S_{v'}$) be a disk containing $i_v(v')$ (respectively $i_v'(v)$) small enough such that its closure does not contain any other point of $Y(f^{k_0})$ than $i_v(v')$ (respectively $i_{v'}(v)$) and that it does not contain $a_v(z)$ (resp. $a_{v'}(z')$). Define 
\[A_n:= (\phi_{n,v}^Y)^{-1}(D)\cap (\phi_{n,v'}^Y)^{-1}(D').\]
We are going to prove the following assertions : 
\begin{enumerate}
\item\label{blle1} for $n$ large enough, $A_n$ is an annulus contained in $\S_n-\{a_n(z),a_n(z')\}$;
\item\label{blle2} $f_n^{k_0}(A_n)$ is contained in $\S_n-\{a_n(z),a_n(z')\}$;
\item\label{blle3}  Every compact sub-set of $D-\{i_v(v')\}$ is included in $\phi_{n,v}(A_n)$ for $n$ large enough and $\phi_{n,w}\circ f_n^{k_0}\circ\phi_{n,v}^{-1}$ converges locally uniformly to $f^{k_0}$ in $D-\{i_v(v')\}$; 
\item\label{blle4} Every compact sub-set of $D'-\{i_{v'}(v)\}$ is contained in $\phi_{n,v'}(A_n)$ for $n$ large enough and $\phi_{n,w'}\circ f_n^{k_0}\circ\phi_{n,v'}^{-1}$ converges locally uniformly to $f^{k_0}$ in $D'-\{i_{v'}(v)\}$. 
\end{enumerate}

\medskip
\noindent{ Point 1.}
Let $M_n:=\phi_{n,v'}\circ \phi^{-1}_{n,v}$. According to Lemma \ref{divrev}, for $n$ large enough, $M_n(\partial D)$ is contained in a neighborhood of $i_{v'}(v)$ that doesn't intersect $\partial D'$. Then, $A_n$ is an annulus. When $n\to \infty$, $\phi_{n,v}\circ a_n(z)$ converges to $a_v(z)$ which is not in the closure of $D$ by assumption. Consequently, for $n$ large enough, $D_n$ does not contain $a_n(z)$. Likewise, for $n$ large enough, $D'_n$ does not contain $a_n(z'_n)$. Then, $A_n\subset \S_n-\{a_n(z),a'_n(z)\}$. 

\medskip
\noindent{ Point 2.}
Let $D_1\subset \S_{v_1}$ be a disk containing $f(D)$ but no other attaching point of any edge of ${\cal T}^Z$ than $f(i_v(v'))$. Let $D'_1\subset \S_{v'_1}$ be a disk containing $f(D')$ but no other attaching point of any edge of ${\cal T}^Z$ than $f(i_{v'}(v))$. Just like in the point 1, for $n$ large enough, 
\[A_{1,n}:=\phi_{n,v_1}(D_1)\cap \phi_{n,v'_1}^{-1}(D'_1)\]
is an annulus.

Given that $\phi_{n,v_1}\circ f_n\circ \phi_{n,v}^{-1}$ converges uniformly to $f$ in the neighborhood of $\partial D$ and that $\phi_{n,v'_1}\circ f_n\circ \phi_{n,v'}^{-1}$ converges uniformly to $f$ in the neighborhood of $\partial D'$, for $n$ large enough we have $f_n(\partial A_n)\subset A_{1,n}$.

Like for the point 1, for $n$ large enough, $A_n$ doesn't intersect $a_n\bigl(Y-]\!]v,v'[\![\bigr)$ so, $f_n(A_n)$ doesn't intersect $a_n\bigl(Z-]\!]v_1,v'_1[\![\bigr)$. Likewise, for $n$ large enough, $A_{1,n}$ doesn't intersect $a_n\bigl(Z-]\!]v_1,v'_1[\![\bigr)$. 

In conclusion, $f_n(\partial A_n)\subset A_{1,n}$ and $f_n(A_n)$ doesn't intersect at least one point in each component of the complement of $A_{1,n}$. It follows from the maximum modulus principle that $f_n(A_n)\subset A_{1,n}$. 

For $n$ large enough, $D-a_v(v')$ (resp. $D'-a_{v'}(v)$) doesn't intersect the set $a_v\bigl(Y(f^{k_0})-]\!]v,v'[\![\bigr)$ (resp. the set $a_{v'}\bigl(Y(f^{k_0})-]\!]v,v'[\![\bigr)$), thus $D_1-a_{F(v)}(F(v'))$ (resp. $D'_1-a_{F(v')}(F(v))$) doesn't intersect $a_{F(v)}\bigl(Y(f^{k_0-1})-]\!]F(v),F(v')[\![\bigr)$ (resp $a_{F(v')}\bigl(Y(f^{k_0-1})-]\!]F(v),F(v')[\![\bigr)$). Thus we can do the same if we replace $D$ and $D'$ by $D_1$ and $D'_1$, the vertices $v$ and $v'$ by $v_1$ and $v'_1$ and iterate this $k_0-1$ times. This proves that $f_n^{k_0}(A_n)$ doesn't intersect $a_n(Z-]\!]v_{k_0},v'_{k_0}[\![)$, in particular $a_n(z)$ and $a_n(z')$. 

\medskip
\noindent{ Points 3 and 4.}
These assertions follow from Lemma \ref{convit}.

\medskip
\noindent{\textbf{Laurent series and convergence.}}

In the charts from $\sigma_{n,v}$ to $\sigma_{n,w}$, the map ${f_n^{k_0}:A_n\to \S_n-\{a_n(z),a_n(z')\}}$ has a Laurent series development:
\[\sigma_{n,w}\circ f_n^{k_0}=\sum_{j\in \Z} c_{n,j}\cdot \sigma_{n,v}^j.\]

In the charts from $\sigma_{n,v'}$ to $\sigma_{n,w'}$, this Laurent series becomes 
\[\sigma_{n,w'}\circ f_n^{k_0}=\sum_{j\in \Z} c'_{n,j}\cdot{\sigma^j_{n,v'}}.\]

As we have 
\[\sigma_{n,v'} = \lambda_n \sigma_{n,v}\quad\text{and}\quad \sigma_{n,w'}=\mu_n\sigma_{n,w},\]
it follows that
\[c'_{n,j} = \frac{\mu_n}{\lambda_n^j} c_{n,j}.\]

Now consider the Laurent series of $f^{ {k_0}}:\S_v\to \S_w$ in the neighborhood of $a_v{z}$ in the charts from $\sigma_{v}$ to $\sigma_{w}$ and that of $f^{ {k_0}}:\S_{v'}\to \S_{w'}$ in the neighborhood of $a_{v'}(z')$ in the charts from $\sigma_{v'}$ to $\sigma_{w'}$
\[\sigma_{w} \circ f^{k_0}= \sum_{j\in\Z} c_j\cdot\sigma_{v}^j\quad, \quad\sigma_{w'}\circ f^{k_0} = \sum_{j\in\Z} c'_j\cdot\sigma_{v'}^j.\]

\medskip 
\noindent{\bf Convergence.}
Here we prove that $$c_{n,j}\underset{n\to \infty}\longrightarrow c_j\text{ and }c'_{n,j}\underset{n\to \infty}\longrightarrow c'_j.$$

As $\F_n\to\F$ and for every $0\leq k<k_0,\;\phi_{n,v_k}^{-1}(\partial D_{k,n})$ doesn't intersect $Z_v$, we have 
$$\sigma_{v_{k+1}}\circ(\phi_{n,v_{k+1}}\circ f_n\circ\phi^{-1}_{n,v_k})\circ\sigma^{-1}_{v_k}\to \sigma_{v_{k+1}}\circ f_{v_k}\circ\sigma^{-1}_{v_k}$$
uniformly on $\sigma^{-1}_{v_k}(\phi_{n,v_k}^{-1}(\partial D_{k,n}))$. So by composition, we have $$\sigma_{w}\circ(\phi_{n,w}\circ f^{k_0}_n\circ\phi^{-1}_{n,v})\circ\sigma^{-1}_{v}\to \sigma_{w}\circ f^{k_0}\circ\sigma^{-1}_{v}$$ uniformly on $\sigma^{-1}_{v}(\partial D)$.

Otherwise, we have $$(\sigma_w\circ\phi_{n,w})\circ f^{k_0}_n\circ(\phi^{-1}_{n,v}\circ\sigma^{-1}_v)=\sigma_{n,w}\circ f^{k_0}_n\circ\sigma^{-1}_{n,v}.$$
The uniform convergence implies the uniform convergence of the coefficients of Laurent series so we have $c_{n,j}\underset{n\to \infty}\longrightarrow c_j$. The proof is the same for $c'_{n,j}\underset{n\to \infty}\longrightarrow c'_j.$

\medskip
\medskip 
\noindent{\bf Conclusions and multipliers.} 

\medskip 
\noindent{\bf Case A.} 
Suppose that the path connecting $z$ to $z'$ in $T^Z$ goes through $w$ and $w'$ in this order. 
On one hand, we have $\sigma_{n,w}=\rho_n\sigma_{n,v}$ with $\rho_n\in \C-\{0\}$ and according to Lemma \ref{noncomp}:
\begin{itemize}
\item $\rho_n\to 0$ if and only if the path connecting $z$ to $z'$ goes through $v$ before going through $w$ and 
\item $\rho_n\to \infty$ if and only if the path connecting $z$ to $z'$ goes through $w$ before going through $v$. 
\end{itemize}
Likewise, 
we have $\sigma_{n,w'}=\rho'_n\sigma_{n,v'}$ with $\rho'_n\in \C-\{0\}$ and
\begin{itemize}
\item $\rho'_n\to 0$ if and only if the path from $x$ to $x'$ goes through $v'$ before going through $w'$ and 
\item $\rho'_n\to \infty$ if and only if the path connecting $x$ to $x'$ goes through $w'$ before going through $v'$. 
\end{itemize}
Note that
\[\rho_n \mu_n\sigma_{n,v} = \mu_n\sigma_{n,w} = \sigma_{n,w'} = \rho'_n\sigma_{n,v'} = \rho'_n \lambda_n\sigma_{n,v}\]
so
\[\rho'_n = \frac{\mu_n}{\lambda_n} \rho_n.\]

On the other hand, the development of $f^{ {k_0}}:\S_v\to \S_w$ in the neighborhood of $z$ in the charts from $\sigma_{v}$ to $\sigma_{w}$ is the one of a function defined in the neighborhood of infinity and mapping infinity to infinity with local degree ${d_0}$. Consequently, $c_j=0$ if $j\geq {d_0}+1$ and $c_ {d_0}\neq 0$. Likewise, the development of $f^{ {k_0}}:\S_{v'}\to \S_{w'}$ in the neighborhood of $z'$ in the charts from $\sigma_{v'}$ to $\sigma_{w'}$ is the one of a function defined in the neighborhood of $0$ mapping $0$ to $0$ with local degree $ {d_0}$. Consequently, $c'_j=0$ if $j\leq  {d_0}-1$ and $c'_ {d_0}\neq 0$.

Under the hypothesis of Lemma \ref{multiplic}, as $d_0=1$ 
\[c'_{n, {1}} = \frac{\mu_n}{\lambda_n^{ {1}}}c_{n, {1}}, \text{ so }\frac{\mu_n}{\lambda_n}\underset{n\to \infty}\longrightarrow\frac{c'_{ 1}}{c_{ 1}}\in \C-\{0\}.\]
It follows that $\rho_n\to 0$ if and only if $\rho'_n\to 0$. Likewise, $\rho_n\to \infty$ if and only if $\rho'_n\to \infty$. 
Given that $[w,w']^Z\subseteq [v,v']^Z$ or $ [v,v']^Z\subseteq [w,w']^Z$, the only possibility is that $\rho_n$ and $\rho'_n$ tend neither to $0$ or to infinity, which implies that $v=w$ and $v'=w'$. 
Then we can chose $\sigma_{v}=\sigma_{w}$ and $\sigma_{v'}=\sigma_{w'}$, which implies that $\lambda_n=\mu_n$ so $c_1=c'_1$. The multiplier of $f^{ {k_0}}$ at $z$ is $1/c_1$ and the multiplier of $f^{ {k_0}}$ at $z'$ is $c'_1$. The product of the multiplier is $1$ as required. 

Under hypothesis of Lemma \ref{annocritiq}, we have $d_0>1$
so \[c'_{n, {d_0}} = \frac{\mu_n}{\lambda_n^{ {d_0}}}c_{n, {d_0}}  \text{ so } \frac{\rho'_n}{\rho_n}\approx \lambda_n^{d_0-1} \left( \frac{c'_{d_0}}{c_{d_0}}\right).\]
Thus if we suppose that $w\in[v,v']$, ie $\rho_n\to\C^\star\cup\{\infty\}$, we can deduce that $\rho'_n\to\infty$, ie $w'\notin[v,v']$. The case $w'\in[v,v']$ is similar.

\medskip 
\noindent{\bf Case B.}
Suppose that the path connecting $z$ to $z'$ in $T^Z$ goes through $w'$ and $w$ in this order. 
On one hand, we have $\sigma_{n,w}=\rho_n\sigma_{n,v'}$ with $\rho_n\in \C-\{0\}$. 
Likewise, 
we have $\sigma_{n,w'}=\rho'_n\sigma_{n,v}$ with $\rho'_n\in \C-\{0\}$.
Note that
\[\rho'_n\sigma_{n,v} = \sigma_{n,w'} =\mu_n \sigma_{n,w} =\mu_n \rho_n\sigma_{n,v'} = \mu_n \rho_n \lambda_n\sigma_{n,v}\]
so
\[\rho'_n = \mu_n \rho_n\lambda_n .\]

On the other hand, the development of $f^{ {k_0}}:\S_v\to \S_w$ in the neighborhood of $z$ in the charts from $\sigma_{v}$ to $\sigma_{w}$ is the one of a function defined in the neighborhood of infinity and that maps infinity to $0$ with local degree ${d_0}$. Consequently, $c_j=0$ if $j\geq {d_0}-1$ and $c_{-d_0}\neq 0$. Likewise, the development of $f^{ {k_0}}:\S_{v'}\to \S_{w'}$ in the neighborhood of $z'$ in the charts from $\sigma_{v'}$ to $\sigma_{w'}$ is the one of a function defined in the neighborhood of $0$ and that maps $0$ on infinity with local degree $d_0$. Consequently, $c'_j=0$ if $j\leq -(d_0+1)$ and $c'_{-d_0}\neq 0$. 

Under hypothesis of Lemma \ref{multiplic}, as $d_0=1$
we have
\[c'_{-1}\underset{n\to \infty}\longleftarrow c'_{n,-1} = \mu_n\lambda^{-1}_n c_{n,-1}\underset{n\to \infty}\longrightarrow c_{-1},\]
so
\[\mu_n\lambda_n\underset{n\to \infty}\longrightarrow\frac{c'_{-1}}{c_{-1}}\in \C-\{0\}.\]
Then $\rho_n\to 0$ if and only if $\rho'_n\to 0$. Likewise, $\rho_n\to \infty$ if and only if $\rho'_n\to \infty$. 
given that $[w,w']^Z\subseteq [v,v']^Z$  or $ [v,v']^Z\subseteq [w,w']^Z$, the only possibility is that $\rho_n$ and $\rho'_n$ tend neither to $0$ or infinity, which according to Lemma \ref{noncomp} implies that $v=w'$ and $v'=w$. 
Then we can chose $\sigma_{v}=\sigma_{w'}$ and $\sigma_{v'}=\sigma_{w}$, which implies that $\lambda_n \mu_n=1$ so the cycle multiplier, ie $c'_{-1}/c_{-1}$ is $1$ as required. 

Under hypothesis of Lemma \ref{annocritiq}, as $d_0>1$
we have
\[ \mu_n\lambda_n=\lambda^{d_0-1} \left(\frac{c'_{n,-d_0}}{ c_{-d_0}}\right)  \text{ so } \frac{\rho'_n}{\rho_n}\approx \lambda_n^{d_0-1} \left( \frac{c'_{d_0}}{c_{d_0}}\right).\]

Thus supposing that $w\in[v,v']$, ie $\rho_n\to\C^\star\cup\{\infty\}$, as $\lambda_n^{d_0-1}\to\infty$ then we have $\rho'_n\to\infty$, ie $w'\notin[v,v']$. The case $w'\in[v,v']$ is similar.


\section{Bi-critical case}\label{chap6}

We now want to study rescaling-limits in the case of rational maps of degree $d$ with exactly two critical points (including for example the case of degree two rational maps). In this case the critical points have exact multiplicity $d-1$. Such maps are said to be bi-critical.

In this subsection we prove the following theorem and conclude with the proof of Theorem \ref{omega}.

 \begin{theoremint2}
Let ${\bf F}$ be a portrait of degree $d$ with $d+1$ fixed points and exactly $2$ critical points and let $({\cal F},{\cal T}^X)\in\partial\Rat_{{\bf F},X}$. Suppose that there exists $(\displaystyle f_n,y_n,z_n)\overset{\lhd}{\longrightarrow}{ \F}$ in $\Rat_{{\bf F},X}$ such that for every $n$, $x_n(X)$ contains all the fixed points of $f_n$.
 Then the map $\F$ has at most two critical cycles of spheres; they have degree $d$.

Assume that there exists at least one rescaling limit. Then there is a vertex $ w_0$ separating three fixed points which is fixed and such that $f_{ w_0}$ has finite order $k_0>1.$
 Denote by $v_0$ the critical vertex separating $w_0$ and the two critical leaves.
\begin{enumerate} 
\item\label{class2}  Either $v_0$ belongs to a critical cycle of period $k_0$ and
\begin{enumerate} 
	\item\label{class2a} its associated cover has a parabolic fixed point;
	\item\label{class2b} if there is a second critical cycle then it has period $k'_0>k_0$, its associated cover has a critical fixed point with local degree $d$ and the cover associated to $v_0$ has a critical point that eventually maps to the parabolic fixed point.
\end{enumerate}
\item\label{class11} Or $F^k(v_0)\notin {\cal T^X}$ for some $k<k_0$; in this case there is exactly one critical cycle; it has period $k'_0> k_0$ and its associated cover has a critical fixed point with local degree $d$.
\end{enumerate}
\end{theoremint2}


Our interest for these maps comes from the following lemma, which makes this case easier to understand.

\begin{lemma}\label{ptscrit}    
Let ${\cal F}:{\cal T}^Y\to {\cal T}^Z$ be a cover between trees of spheres. Every critical vertex lies in a path connecting two critical leaves. Each vertex on this path is critical.
\end{lemma}

\begin{proof} Let $v$ be a critical vertex of ${\cal F}$. Then $f_v$ has at least two distinct critical points. There are at least two distinct edges attached to $v$. So $v$ is on a path of critical vertices.

Let $[v_1,v_2]$ be such a path with a maximal number of vertices. From this maximality property, we see that there is only one critical edge (edge with degree strictly greater than one) attached to $v_1$. If $v_1$ is not a leaf then $f_{v_1}$ has just one critical point and that is not possible. So $v_1$ is a leaf. Similarly, $v_2$ is a leaf. 
\end{proof}

\begin{notation}
We will denote by  $c$ and $c'$ the two critical leaves. Then according to Lemma \ref{ptscrit} 
the critical vertices are the vertices of $C_d:=[c,c']$. We will use the notation $c_k:=F^k(c)$ and $c'_k:=F^k(c')$ when they are well defined.
\end{notation}

\medskip
\noindent{\bf The fixed vertex $w_0$ in Theorem \ref{class0}.}

We will call a principal branch every branch attached to a critical vertex in $\T^Y$ with vertices of degree one and containing a fixed leaf. We prove that there are two fixed leaves lying in the same principal branch and that the vertex separating them and $C_d$ is fixed. This corresponds to the Case 1a below for which we can prove this lemma. The others cases are not absurd.


{\bf Case1.} At least two fixed leaves are not critical.

{\bf Case1a.} Two non critical leaves $\alpha$ and $\beta$ lie in the same principal branch. We will see that this is the only possible case. 

Denote by $w_0$ the vertex separating $\alpha$, $\beta$ and a critical leaf. Then $w_0$ is not critical and both of the critical leaves lie in the same branch on it. Let $v_0$ be the vertex separating $w_0$ and the two critical leaves. Then $A:=[w_0,v_0]\cup C_d$ maps bijectively to its image which is inside a branch $B$ on $F(w_0)$. 
The iterates of $w_0$ lie in $]\alpha,\beta[$. Suppose that $w_0$ is not fixed and for example that $w_0$, $F(w_0)$ and $\beta$ lie on $]\alpha,\beta[$ in this order. Then the only pre-image of $a_{F(w_0)}(\alpha)$ is $a_{w_0}(\alpha)$. It follows that $F(B_{w_0}(\beta))\subset B_{w_0}(\beta)$ and that $B\subset B_{w_0}(\beta)$ so $B$ cannot contain any critical periodic vertex. Thus $F(w_0)=w_0.$ 
 
 As $w_0\in [\alpha,\beta]$, it follows that $f_{w_0}$ is conjugate to a rotation. If it doesn't have finite order, then the iterates of $B$ are branches attached at the iterates of $f_{w_0}(a_{w_0}(v_0))$ which are all distinct and contains the iterates of $C_d$ which is absurd. So $w_0$ has finite order. 

{\bf Case1b.} All the critical leaves lie in different principal branches. Let $\alpha$ and $\beta$ be two critical leaves. Then there exists a critical vertex $w_0$ in $]\alpha,\beta[$, the closest to $\alpha$. The vertex $w_0$ is not fixed. We can prove that there exists a fixed leaf $\gamma$ such that ${B_{F(w_0)}(\gamma)\subset B_{w_0}(\gamma)}$. Indeed, suppose that there is no $\gamma$.
For every fixed leaf $\delta$, the point $a_{w_0}(\delta)$ is a preimage of $a_{F(w_0)}(v_0)$.  Either one of the fixed leaves $c$ is critical so the $a_{w_0}(c)$ the $a_{w_0}(\alpha) (\neq a_{w_0}(c))$ would be $d+1$ preimage of $a_{F(w_0)}(v_0)$ counting with multiplicities. Or all the fixed leaves lie in different branches on $w_0$; thus as $f_{w_0}$ is the uniform limit of $\phi_{n,F(w_0)}^{-1}\circ f_n\circ \phi_{n,w_0}$ in the neighborhood of the $a_{w_0}(\delta)$, the the $a_{w_0}(\delta)$ would be again exactly $d+1$ preimage of $a_{F(w_0)}(v_0)$ counting with multiplicities.
In both cases we obtain a contradiction.

Let $B_0:=B_{w_0}(\gamma)$. We have $F(B_0)\subset B_0$.
Then $F(C_d)\subset B_0$ because if not $c_0$ is a critical leaf then $$f(a_{w_0}(c_0))=a_{F(w_0)}(w_0)=F(a_{w_0}(\alpha))$$ so $a_{F(w_0)}(w_0)$ would have at least $d+1$ preimages counted with multiplicities which is absurd. It follows that every periodic vertex lies in $B_0$.

Suppose that $v$ is a periodic critical internal vertex. 
If $F(w_0)\notin C_d$ then $v\notin B_0$ but $F(v)\in B_0$. But $F(B_0)\subset B_0$ so $v$ cannot be periodic which is absurd. Thus $F(w_0)\in C_d$ and it follows that $B_v(F(v))$ maps to a branch on $F(v)$ so maps inside itself. But $F(v)\neq v$ so $v$ cannot be periodic which is absurd.

{\bf Case2.} There is just one non critical fixed leaf $\alpha$. Then the vertex separating $\alpha$ and the two critical leaves is fixed and has maximal degree so, according to Lemma \ref{cvuutil}, the approximating sequence converges in $\rat_d$. Contradiction.

{\bf Case3.} All fixed leaves are critical. Then $C_d$ maps bijectively to itself. Denote by $v$ the periodic critical internal vertex. As the approximating sequence does not converge in $\rat_d$ then  $v$ is not fixed. From the critical annulus lemma (\ref{annocritiq}) $F^2(v)\notin [v,F(v)]$. As  $C_d$ maps bijectively to itself the order of the vertices $v,F(v)$ and $F^2(v)$ has to be preserved so $v$ cannot be periodic. This is absurd.


\begin{notation}
The vertex $w_0$ has a finite order that we denote by $k_0$. The vertex $w_0$ is not critical, so it doesn't lie in $C_d$. We denote by $v_0\in T^X$ the vertex separating $c,c'$ and $w_0$. 
\end{notation}

\begin{remark}\label{rembij} 
 Suppose that $v'$ is a critical periodic internal vertex. Denote by $ B^Y_0$ the branch on $w_0$ containing the two critical points (so $v'$ too) and for $1\leq k\leq k_0$, denote by $ B_k^\star\in IV^\star$ the branch attached to $a^\star_{w_0}(F^k(v'))$. We have $F: B^Y_k\to  B^Z_{k+1}$ is a bijection.
\end{remark}

\medskip
\noindent{\bf Point \ref{class2a} of Theorem \ref{class0}.} Here we suppose that $v_0$ is not forgotten by $F_{k_0}$. 
The annulus $A_0:=]\!] { w_0},{v_0}[\![$ does not contain vertices of degree $d$ so from Corollary \ref{anneaudeg1} we know that $F$ is injective on $\overline{A}_0$. Moreover, elements of $C_d$ have maximum degree,  so $F$ is injective on $A_0\cup C_d$. 
According to remark \ref{rembij}, $F$ is bijective from $ B^Y_k$ to $ B^Z_{k+1}$ for $1\leq k< k_0$, thus, the vertex $v_k:=F^k({v_0})\in B^Z_k $ so $v_{k_0}\in  B_0$.

We have three cases: $v_{k_0}\in[ { w_0},{v_0}], {v_0}\in[ { w_0}, v_{k_0}]$ or $v_{k_0}\in A_0-[ { w_0},{v_0}]$.
In the first two cases, as we know that $F^k(A_0\cup C_d)$ for $1\leq k< k_0$ is included in $ B_k$ which contains only vertices of degree 1, then the non critical annulus lemma (\ref{multiplic}) applied to $F^k(A_0)$ ensure that $v_{k_0}={v_0}$. Hense the vertex ${v_0}$ is periodic with period $k_0$. It is the only vertex of degree $d$ in its cycle. As $a_{w_0}(v_0)$ is a fixed point of $f^{k_0}$ with multiplier 1, according to the same lemma the rational map $f^{k_0}:\S_{v_0}\to \S_{v_0}$ has a parabolic fixed point at $i_{{v_0}}( { w_0})$. According to Corollary \ref{cvuutil} we have $k_0>1$.

 \begin{figure}
  \centerline{\includegraphics[width=9cm]{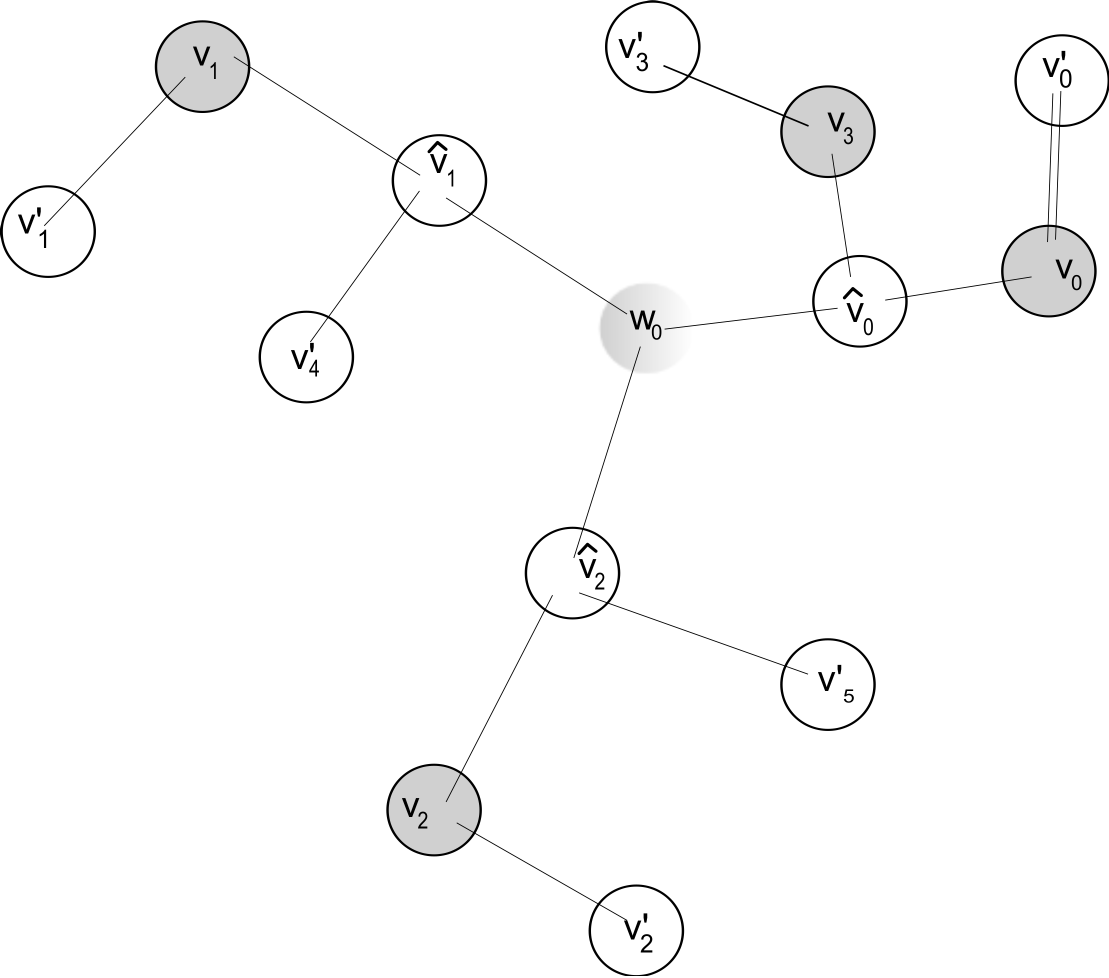}}
   \caption{Example (simplified) of the imposible case of the proof of point \ref{class2} with $k_0=3$}
\label{absurde} \end{figure}

To finish the proof we show that this third case is absurd. 
Indeed, in this case, $v_{k_0}, {v_0}$ and $ { w_0}$ are not on a common path. So there is a vertex $\hat v_0\in [w_0,v_0]$ that separates $v_{k_0}$, $v_0$ and $w_0$. Otherwise, every critical cycle of spheres has to correspond to a cycle of vertices that intersects $C_d$ because by definition a critical cycle of spheres doesn't have degree 1. Let $v'_0$ be a vertex in this intersection. The vertices $v'_0$ and ${v_0}$ lie on a same branch of $\hat {v_0}$ disjoint from the one containing $v_{k_0}$. 
As $A_0\cup C_d$ and its iterates map bijectively to their image, we know that $v_{k_0}$ and $v'_{k_0}$ lie on a same branch of $\hat {v_0}$ so $\hat {v_0}$ separates the vertices $ { w_0}, v'_0$ and $v'_{k_0}$. These four vertices lie in $A_0$ and are not forgotten by $F^{k_0}$, so according to Lemma \ref{definiX} the $k_0$ iterates of $\hat {v_0}$ are well defined. 

As $A'_0:=]\!] { w_0},\hat {v_0}[\![\subset A_0$, its $k_0$ iterates map bijectively to their images which are the $]\!] { w_0},\hat {v_k}[\![$ which all contain vertices of degree one except maybe $]\!] w_0,\hat v_{k_0}[\![$. 
But $w_0$, $\hat v_0$ and $v_0$ are aligned in this order, so for $0\leq k\leq k_0$, it is the same for $w_0$, $\hat v_k$ and $v_k$ then both $\hat v_0$ and $\hat v_{k_0}$ lie on the path $[w_0,v_{k_0}]$. So we have the inclusion
$]\!] w_0,\hat v_0 [\![\subset]\!] w_0,v_{k_0} [\![ $ or the inverse inclusion. 
Thus according to the non critical annulus lemma (\ref{multiplic}) we have $\hat {v_0}=\hat v_{k_0}$. Then the situation is similar to the one in Figure \ref{absurde}.

As we did in previous cases for ${v_0}$, we can prove that the cover associated to the cycle containing $\hat {v_0}$ has a fixed point with multiplier 1 using the non critical annulus lemma. As $\hat {v_0}$ has degree 1, this cover in a projective chart is the identity or a translation. If it is a translation then the branch $B_{\hat {v_0}}(v'_{k_0})$ would be of infinite orbit which contradicts the existence of $v'_{k_0}$ lying on it  and which is periodic. If it is the identity then $F(B_{\hat {v_0}}(v'_{k_0}))\subset B_{\hat {v_0}}(v'_{k_0})$, this contradicts the fact that $v'_0$ lies in the orbit of $v'_{k_0}$. So it is again absurd.

\medskip
\begin{notation}
As in Figure \ref{Rescaling3}, we define $D^\star_i:=T^\star- D^\star_{v_i}(v_0)$. Denote by ${\alpha:=a_{v_0}(w_0)}$ the parabolic fixed point and $\beta_1,\ldots,\beta_{d-1}$ its preimages by $f^{k_0}$. If there is a branch in $T^Y$ attached to $\beta_i$ on ${v_0}$, we denote it by ${B'_i}$. If it is not the case, we set ${B'_i}=\emptyset$.
 \end{notation}

\begin{remark}\label{remresc}
As $D^\star_k\subset  B^\star_k$, according to remark \ref{rembij} we have  $F:D^Y_k\to D^Z_{k+1}$ is a bijection for $1\leq k<k_0$. Because $\alpha$ is not critical, the ${B'_i}$ contain only elements of degree 1 so from Corollary \ref{branchdeg1}, we know that $F({B'_i})$ is the branch on $v_1$ attached to $f(\alpha)$. This means $F({B'_0})=B_{v_0}(w_0)$. We deduce that ${B'_i}\neq\emptyset$. Moreover, as $F:T^Y\to T^Z$ is surjective, we have $F(D_0-{B'_i})=D_1$.
\end{remark}

\begin{remark} As $k_0>1$ we can deduce from the last remark that there are no critical fixed leaf.
\end{remark}

\begin{corollary}\label{autreptcrit}
There exists a critical leaf $c_0$ such that $B_{v_0}(c_0)$ does not contain any critical spheres cycle.
\end{corollary}

\begin{proof}
The orbit of a critical point $z_0$ lies in  the basin of the parabolic fixed point $f(\alpha)$ of $f^{k_0}:\S_{v_0}\to \S_{v_0}$ (cf \cite{DynInOne} for example). The branch $B$ of $T^Y$ on $ v_0$ corresponding to this critical point contains a critical leaf $c_0$. 

If $B$ contains a periodic vertex, then its iterates are defined as soon as they are branches.
Iterates of $a_{v_0}(c_0)$ do not intersect $\alpha$, because this one is pre-fixed so all the iterates of $B$ are branches. Indeed, either they lie in $D_0-{\bigcup B'_i}$ and we can apply Corollary \ref{attachbranch}, or they lie in the $D_i$ with $i>0$ and then we have Corollary \ref{branchdeg1}. From this we deduce by Lemma \ref{branchcrit} that $B$ does not contain critical spheres cycle.
\end{proof}

\begin{notation}From now, we will denote by $c'$ this critical leaf and $c$ the other.
\end{notation}

 \begin{figure}
  \centerline{\includegraphics[width=15cm]{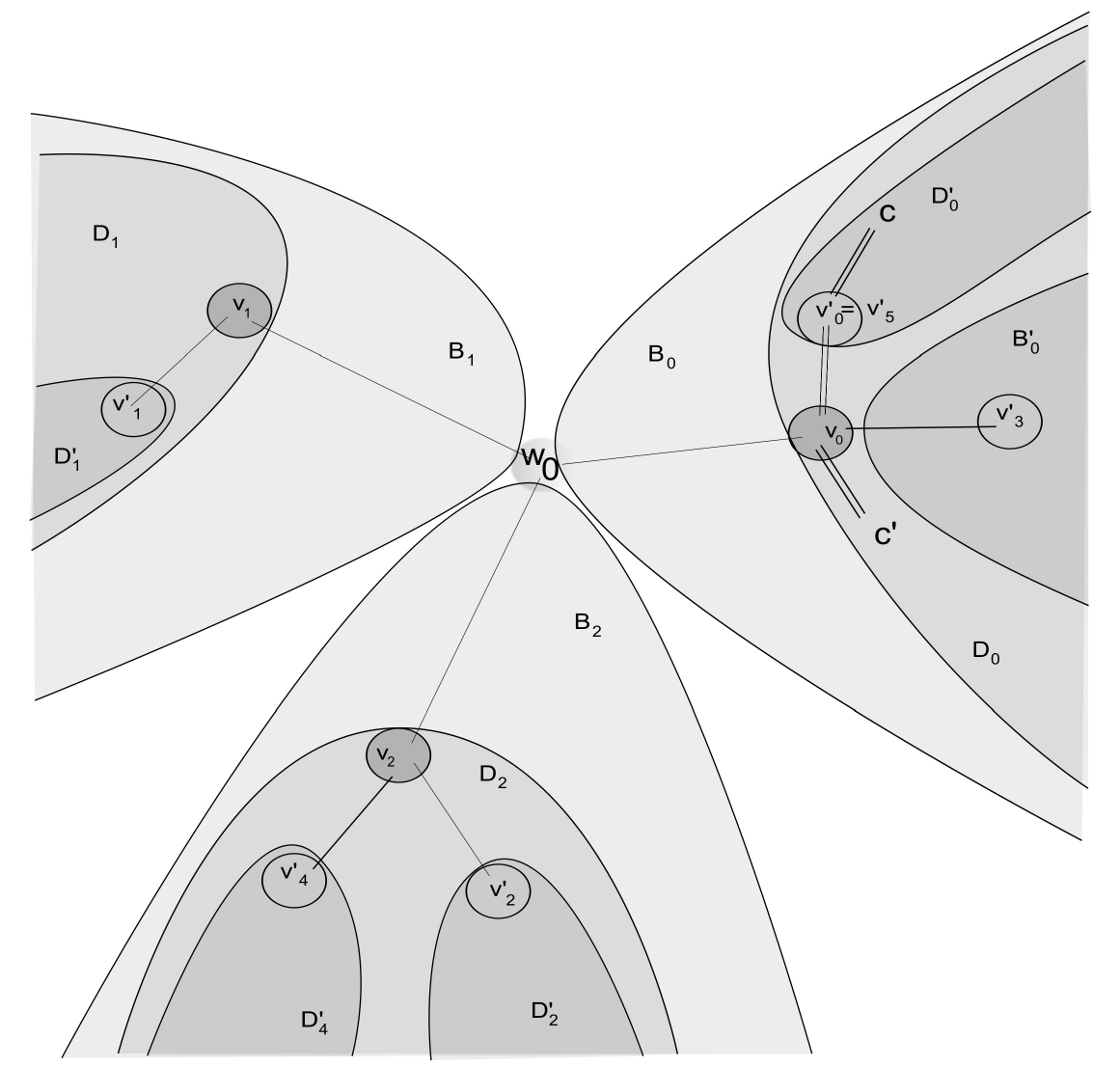}}
   \caption{Simplified representation of a tree $\T^X$ for an example of a cover ${\F:\T^Y\to \T^Z}$  limit of degree $2$ rational maps having a two critical spheres cycle. One of these has period 3 and the other period 5 with the notation introduced in the proof of Theorem \ref{class0}.}
\label{Rescaling3} \end{figure}

\medskip
\noindent{\bf Point \ref{class2b} of Theorem \ref{class0}.}

Every critical spheres cycle has at least degree $d$ so it has a vertex $v'_0$ in $C_d$. We define $v'_k:= F^k(v'_0)$ and $k'_0$ to be the period of this cycle.

Let $D'_k:=T^Y- B_{v'_k}(w_0)$.
Let prove that $F(D'_k)=D'_{k+1}$ and ${f(a_{v'_k}(w_0))=a_{v'_{k+1}}(w_0)}$. According to remark \ref{remresc}, it is true for every  $k$ when  $v'_k \notin {B'_0}$ and we always have $f(a_{v'_k}(w_0))=a_{v'_{k+1}}(w_0)$. Suppose that there exists $i$ such that $v'_k \in {B'_i}$. As we have ${D'_k=B_{v_0}(v'_k)-]\!] v_0,v'_k[\![}$ and $F$ is bijective on $B_{v_0}(v'_k)$, we deduce that $$F(D'_k)=B_{v_1}(v'_{k+1})-]\!] v_1,v'_{k+1}[\![=D'_{k+1}.$$
 Moreover $F$ is a bijection between the edges of $v'_k$ and the one of $v'_{k+1}$. We deduce that $f(a_{v'_k}(w_0))=a_{v'_{k+1}}(w_0)$. 
 
From $F(D'_k)=D'_{k+1}$ we conclude by Lemma \ref{branchcrit} that $D'_0$ does not contain any critical periodic internal vertex so after supposing that $v'_0$ is the degree $d$ sphere of the cycle the closest to $ { w_0}$, we deduce that $v'_0$ is the only critical vertex of the cycle. Thus the associated cover has degree 1. As $f(a_{v'_k}(w_0))=a_{v'_{k+1}}(w_0)$ and $a_{v'_{0}}(w_0)$ is critical, this cover is conjugate to a degree $d$ polynomial.

Now we prove that $k'_0>k_0$.
As $f_{{v_0}}$ has degree $d$, a critical point of degree $d$ cannot be a preimage of the parabolic fixed point. So $v'_0\in D_0- {\bigcup B'_i}$. Thus according to remark \ref{remresc} the $v'_k$ lie in the $D_k$ for $0\leq k\leq k_0$ so $k'_0\geq k_0$. If there is equality then we have $F(]\!] { v_0},v'_0[\![)=]\!]F^k(v_0),F^k(v'_0)[\![=]\!]v_0,F^k(v'_0)[\![$ which contradicts Lemma \ref{annocritiq}.

Suppose that one of the iterates of $v'_0$ lies in some ${B'_i}$. Let $v'_i$ be the first iterate of $v'_0$ in ${B'_i}$. According to the above, for $0\leq k< k+k_0\leq i$ we have ${F^{ k_0}:]\!]{v_0},v'_k[\![\to ]\!]{v_0}, v'_{k+k_0} [\![}$ is a bijection. Thus 
$$f^{k_0+i}\circ a_{{v_0}}(v'_0)=f^{k_0}\circ f^i\circ a_{{v_0}}(c_0)=f^{k_0}\circ a_{{v_0}}(v_i)=f^{k_0}(\beta_i)=\alpha.$$ 
So a preimage of the parabolic fixed point is a critical point.

Let us prove by contradiction that  one of the iterates of $v'_0$ lies in some ${B'_{i_0}}$, which will finish the proof.
If it is not the case, we can apply Lemma \ref{branchcrit} 
 to the branch $B:=B_{v_0}(c)$ because the iterates of $B$ lie in the $D_i$ for $i\neq 1$ or in $D_0-{\bigcup B'_i}$ according to remark \ref{remresc} and because these iterates are disjoint to $B_{v_0}(c')$.

\medskip
\noindent{\bf End of point \ref{class2} (Number of rescalings).}

We prove that in the case \ref{class2}, there are at most two critical cycles of spheres.
We have $C_d=[c',v_0]\cup[v_0,v'_0]\cup[v'_0,c]$. But $[c',v_0]-\{v_0\}$ does not contain periodic vertices. According to Corollary \ref{autreptcrit},  it is the same for $[v_0,v'_0]-\{v_0,v'_0\}$ by definition of $v'_0$ and $[v'_0,c]$ because $[v'_0,c]\subset D'_0$ hence $C_d$ contains only two periodic vertices.

\medskip
\noindent{\bf Point \ref{class11} of Theorem \ref{class0}.}

First recall the following lemma from \cite{A2}.

\begin{lemma}\label{ext22} Suppose that $\displaystyle { f}_n\overset{\lhd}{\longrightarrow}{ \F}$ and $z\in Z\setminus X$
then after passing to a subsequence there exists extensions $(f_n, y_n, z_n)_n\lhd(\tilde f_n,\tilde y_n,\tilde z_n)_n$ with $z\in \tilde X$ and $\forall n\in\N, \tilde x_n(z)=z_n(z)$ and $\tilde \F$ such that $\displaystyle {\tilde f}_n\overset{\lhd}{\longrightarrow}{\tilde \F}$  and 
\begin{itemize}
\item $\T^X\lhd\T^{\tilde X}$, $\T^Y\lhd\T^{\tilde Y}$,
and $\T^X\lhd\T^{\tilde Z}$,
\item $\forall v\in IV^{ Y}, F(v)\in T^X\implies \tilde f_v=f_v$.
\end{itemize}
\end{lemma}

Denote by $v'$ a periodic critical internal vertex.
According to this lemma, 
after considering a subsequence, we can find such a $(\tilde \F,\T^{\tilde X})$ such that ${ \tilde X}$ contains the first $k_0$ iterates of $c'$ (where $c'$ is the critical leaf such that $v'\notin]c',v_0[$). Then we can show for $k$ from $1$ to $k_0$ that the vertex $\tilde F^k(v_0)$ separate the vertices $w_0$, $\tilde F^k( c')$ and $\tilde F^k(v')$. It follows from Lemma \ref{definiX} that $v_0$ is not forgotten by $\tilde F^{k_0}$. Then we conclude by applying point \ref{class2} and using the fact that in this case there can be just two different critical cycles of spheres. It follows again that in this case there cannot be more than two critical cycles of spheres.

\begin{remark}We just proved that with the hypothesis of Theorem \ref{class0}, if there exists a critical cycle of spheres with an associated polynomial cover then, after passing to a subsequence, we are in case \ref{class2}.
\end{remark}



\medskip
\noindent{\bf Proof of Theorem \ref{omega2}.}

First recall the following theorem from \cite{A2}.

\begin{theorem}\label{proparbrplin}
Given a sequence $(f_n)_n$ in $\Rat_d$ for ($d\geq2$) with $p\in \N^*$ classes $M_1,\ldots,M_p$ of rescalings.
Then, passing to a subsequence, there exists a portrait ${\bf F}$, a sequence ${(f_n,y_n,z_n)_n\in\Rat_{{\bf F},X}}$ and a dynamical system between trees of spheres $(\F,\T^{ X})$ such that
\begin{itemize}
\item $\displaystyle { f}_n\overset{\lhd}{\underset{\phi_n^Y,\phi_n^Z}\longrightarrow}{ \F}$ and
\item $\forall i\in[\!\![1,p]\!\!], \exists v_i\in\T^Y$, $M_i\sim(\phi^Y_{n,v_i})_n$.
\end{itemize}
\end{theorem}

Take a sequence of bi-critical maps in $Rat_d$ and suppose that it admits $p\geq 2$ dynamically independent rescalings of period at least $2$. Applying this theorem, passing to a subsequence we obtain such a $(\F,\T^{ X})$. Then according to Lemma \ref{ext22} we can suppose after passing to a subsequence that the portrait ${\bf F}$ satisfies the hypothesis of Theorem \ref{class0}. Then the conclusions follow immediately.


\section{The case of degree $2$}\label{last}

\subsection{Example}\label{decompmilnor}

In this part we propose to understand a concrete example of how to compute rescaling limits in the case of degree 2. Theorem \ref{class0} tells us how to find them but not how to prove that they exist.

In \cite{RemarksQuad} J. Milnor notices that we have the following surprising relation : the set of rational maps of degree 2 having a period 2 cycle with multiplier $-3$ (denoted by $Per_2(-3)$) has always a period 3 cycle with multiplier 1 (so is included in $Per_3(1)$).

During a MRC program session in June 2013 organized in Snowbird, Laura De Marco and Jan-Li Lin tried to understand this decomposition and studied the family $Per_2(-3)$. We have a parametrization (not injective) of this family given by
$$f_a := \frac{(1+3a)(-a+z)}{(1-a)(3az+z^2)}.$$
Here, the 2-cycle with multiplier $-3$ is  $\{0,\infty\}$. When $a\to 1$, the family $f_a$ diverges and $[f_a]$ diverges in $rat_2$.

In this case we can find two rescaling limits and the knowledge of the limiting dynamics gives us the good rescalings that we have to look for in order to compute them.

 As we have a persistent period 2 cycle that converges when $a\to1$, we know that for some normalization the second iterate of $f_a$ converges to a quadratic rational map with a parabolic fixed point and that it separates critical points. As critical points of $f_a$ are $-a$ and $3a$ and converge to different limits, we are in this normalization. After computations we verify that 
 $$ f^2_a\to f^2_1:=z(3+z)/(z-1) \text{ when } a\to 1\quad\text{(cf Figure \ref{ExMilnor2})}.$$
 
  \begin{figure}
 \centerline{\includegraphics[width=8.5cm]{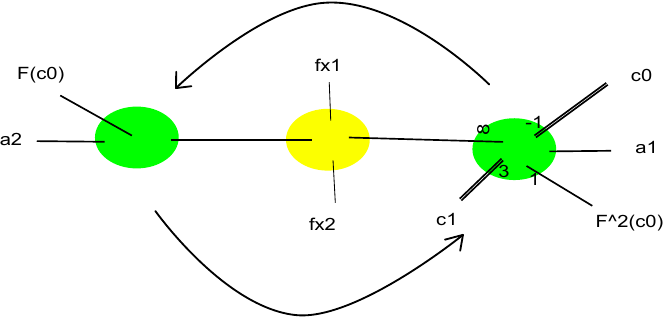}}
   \caption{The yellow sphere is fixed and the associated cover is $-Id$. The green spheres have period 2 and  are adjacent to a period 2 cycle ($a1\to a2\to a1$) with multiplier $-3$. On the green sphere on the right, the rescaling limit has a parabolic fixed point at infinity.}
\label{ExMilnor2} \end{figure}

Infinity is a parabolic fixed point for $f^2_1$. We can also verify that the fixed sphere is in the branch attached at infinity (two of the three fixed points of $f_a$ converge to infinity). The third fixed point is constant equal to $1$ which is the second preimage of infinity for $f^2_1$. Otherwise we note that $-1$ is a critical point of $f^2_1$ and that $f^2_1(-1)=1$ so $-1$ is prefixed (so the other critical point $3$ lies in the parabolic basin of infinity). So there is no contradiction for the existence of a second resealing limit. J.Milnor's remark suggest to look if there is not a rescaling limit of period 3 marked by  a period 3 cycle. We would be in the configuration of Figure \ref{exmil}.

 \begin{figure}
 \centerline{\includegraphics[width=12.5cm]{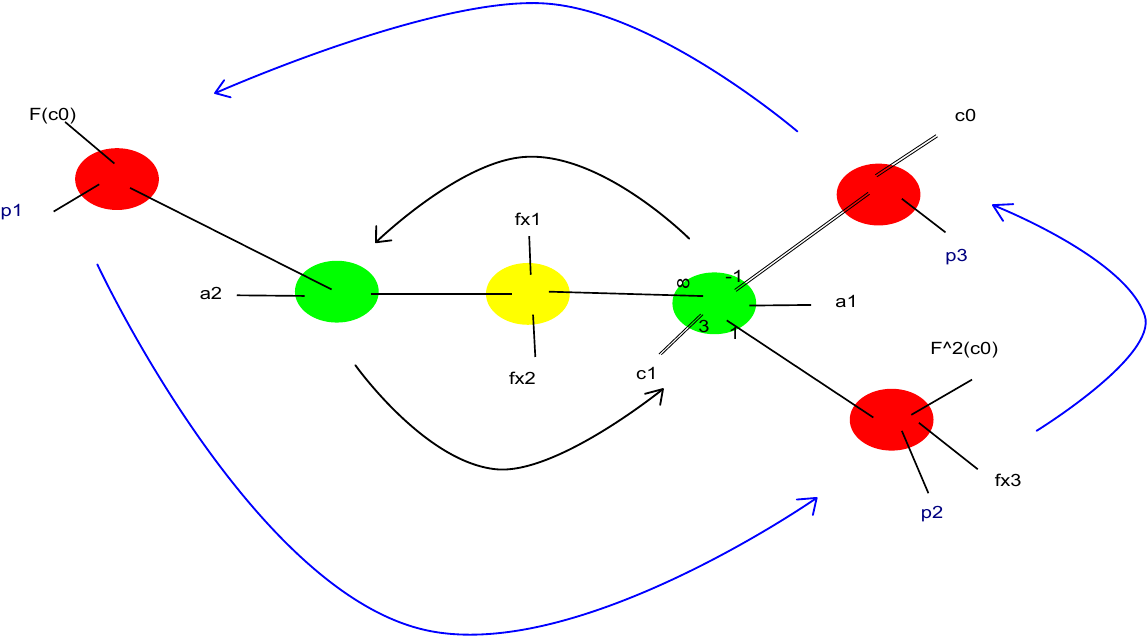}}
   \caption{The yellow sphere is fixed and the cover associated is $-Id$. The green spheres have period 2 and are adjacent to a period 2 cycle ($a1\to a2\to a1$) with multiplier $-3$. On the green sphere on the right, the rescaling limit has a parabolic fixed point at infinity, $c1$ is in its direct basin and $c0$ maps after two iterates to infinity. The red spheres are a cycle of period 3 and are adjacent to a period three cycle with multiplier 1 ($p1\to p2\to p3 \to p1$). The rescaling limit associated to the upper-right sphere is a quadratic polynomial.}
\label{exmil} \end{figure}

We know that such a sphere would be marked by the critical point tending to $1$ and an element of this orbit.
Thus we are looking for a point of period three $f_a$ that tends to $-1$ and after a computation we see that there exists exactly one like it that we will denote by $p3_a$.

We choose the Moebius transformation $M_a$ such that $$M_a(3a)=\infty, M_a(-a)=0\text{ and }M_a(p3_a)=1,$$ such that:
$$M_a(z)=\frac{z+a}{z-3a}\cdot \frac{p3_a-3a}{p3_a+a}~.$$
Thus if such a rescaling limit exists, then the branch containing one of the critical points would be at infinity and fixed and the other one would be at $0$, so we would obtain a quadratic polynomial of the form $z^2+c$ with $c\in \C$. After computation we find:
$$\lim_{a\to 1}M_a\circ f^3_a\circ M_a^{-1}(z)=z^2+1/4.$$

\begin{remark}
We can see that in this case, if we didn't mark the cycle $a_1,a_2$ then we are in the case \ref{class11} of Theorem \ref{thmkiw2}. In Figure \ref{ExMilnor2} we are in the case \ref{class2a} of Theorem \ref{thmkiw2} and in Figure \ref{exmil}, in the case  \ref{class2b} of Theorem \ref{thmkiw2}.
\end{remark}


\subsection{Comparison with J.Milnor's compactification}\label{last2}

\noindent{\textbf{J. Milnor's point of view.}}

In \cite{RemarksQuad}, J. Milnor provides a parametrization of $\rat_2$ by looking at two of the symmetric functions of the multipliers at the fixed points. In such a way, $\rat_2$ can be viewed as a subset of $\C^2$ and be compactified as a subset of $\C P^2$.

Consider a sequence of degree 2 rational maps $(f_n)_n$.
Each $f_n$ after a conjugacy by a Moebius transformation is of the form 
$$f_n(z)=z\frac{z+\alpha_n}{\beta_n z+1} \quad\text{ with } \alpha_n\beta_n\neq1,$$
where $\alpha_n$ and $\beta_n$ are the respective multipliers of the fixed points $0$ and $\infty$. 

Denote by $\gamma_n$ the multiplier of the third fixed point.
If $(f_n)_n$ diverges, one of the multipliers diverges. Suppose for example that $\gamma_n\to \infty$.
The index formula assures that $$\frac{1}{1-\alpha_n}+\frac{1}{1-\beta_n}+\frac{1}{1-\gamma_n}=1\quad\text{ so }\alpha_n\beta_n\to 1.$$
Suppose that $\beta_n\to\beta_\infty\in\C^\star$.
Then $f_n\to(\beta_\infty \cdot Id)$ locally uniformly outside a point which is the limit of the critical points of the $f_n$. J. Milnor proved that the intersection points between the boundary of $\rat_2$ and the curves corresponding to rational maps with a cycle of given period and multiplier are the points where two of the multipliers are conjugate roots of unity.

\medskip
\noindent{\textbf{Our point of view on $\rat_2$.}}

 \begin{figure}
  \centerline{\includegraphics[width=11cm]{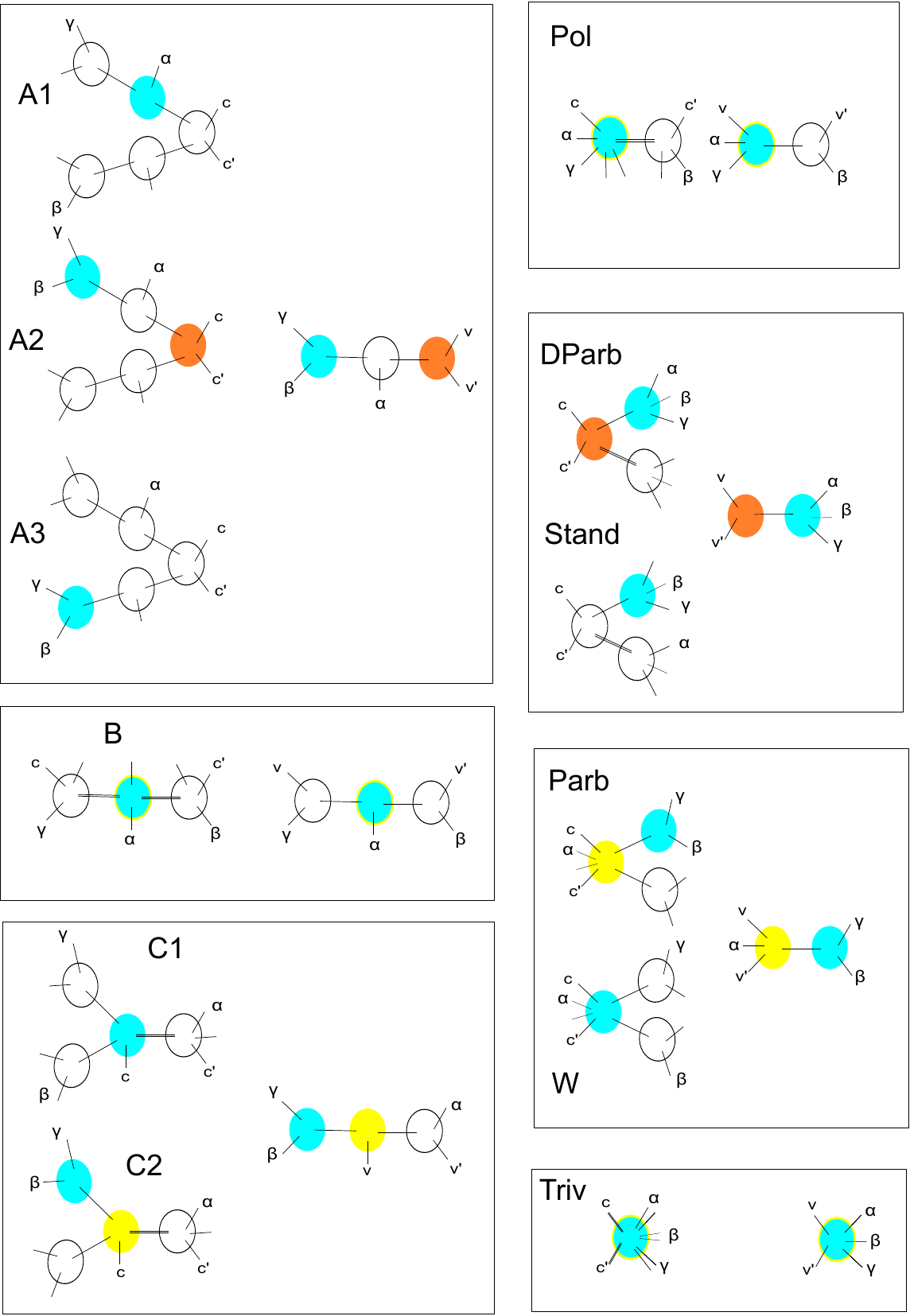}}
   \caption{Enumeration of the possible configurations (after permutations of the labelings) of the covers between trees of spheres with a portrait corresponding to a rational map $f$ of degree 2 with non super-attractive fixed points $\alpha,\beta$ and $\gamma$  and two critical points $c$ and $c'$. We use the notation $v:=f(c)$ and $v':=f(c)$.  The tree $T^Y$ is on the left and the corresponding tree $T^Z$ is on its right. The pre-fixed leaves of $F$ are not labeled. 
   }
\label{tarbr3fix} \end{figure}

Consider a rational map $f$ of degree 2 with 3 distinct and non super-attractive fixed points $\alpha,\beta,\gamma$. Let $X=\{ \alpha,\beta,\gamma\}$. By the Riemann-Hurwitz formula we know that $f$ has exactly two  critical points, that we will denote by $c$ and $c'$, so two critical values $v:=f(c)$ and $v':=f(c')$. We set $Z:=X\cup \{v,v'  \}$ and $Y:=f^{-1}(Z)$. We define ${\bf F}:=(f|_Y,deg_f|_Y)$ to be the corresponding portrait.

Consider the set of rational maps of degree 2 with 3 distinct and non super-attractive fixed points.  We can define three injections $x,y$ and $z$ such that these rational maps are marked by $(f,y,z)$ and such that we have $y|_X=z|_X$. All these rational maps have same portrait ${\bf F}$ and we will denote it by $\Rat_{{\bf F},X}$. 
Note that $\Rat_{{\bf F},X}\subset\Rat_2$ but this inclusion is strict. Indeed, this set does not contain:
\begin{itemize}
\item the rational maps with a critical fixed points (conjugated to a polynomial in $\rat_2$);
\item the rational maps with a simple parabolic point;
\item the rational maps with a double parabolic point.
\end{itemize}

We are going to see that all of these missing elements appear in some way in $\partial\Rat_{{\bf F},X}$.

We take $\F$ a dynamical system between trees of spheres with portrait ${\bf F}$.
The cover $\F$ has degree 2. According to Lemma \ref{ptscrit}, 
$\F$ has two critical leaves $c$ and $c'$ and all of the other critical vertices lie on the path connecting them. Figure \ref{tarbr3fix} represents all the different possibilities of combinatorial trees for such a $\F$ (after a change of the fixed or critical leaves labels).

The vertex $w_0$ separating the three fixed leaves is represented in cyan. It is surrounded by some yellow when it is critical and fixed, ie in the configurations {\bf B, Pol} and {\bf Triv}.  According to Corollary \ref{cvuutil}, the covers are converging in $\rat_2$ in those cases. In the cases {\bf Pol} and {\bf Triv}, we recognize the limits which are respectively the class of the polynomial maps and of the rational maps which have no super-attractive fixed points and no parabolic fixed points. In the case {\bf B}, we recognize the class of the polynomial maps with a super -attractive fixed point, ie the class of $z\to z^2$.

The vertex $w_0$ is not fixed in the configurations {\bf A1,C1} and {\bf W} so there is no rescaling limits in these cases.

Denote by $v_0$ the critical vertex which is the closest to $w_0$.

In the configurations {\bf C2} and {\bf Parb}, the vertex $v_0$ and its image are on the path $[w_0,\alpha]$, so we can apply the annulus lemma in the non critical case and conclude that $v_0$ is fixed, thus according to Corollary \ref{cvuutil} we are in the closure of $\rat_2$. In the case {\bf C2} we remark that the critical point which is the attaching point of the branch of $c'$ on $v_0$ is fixed and it is the limit of a fixed point so we are in the case of the polynomial maps  class and there is in addition a double fixed point so this polynomial is conjugated to $z\to z^2+1/4$. In the case {\bf Parb} there is again a double fixed point and a third non critical fixed point so we are in the parabolic rational map class which have no super-attractive fixed point.

The configurations {\bf DParb, A2} and {\bf A3} are a bit more complicated to identify. For this suppose that these covers are dynamical limits of dynamical systems between spheres covers with portrait ${\bf F}$. Then according to the last lemma of {\cite{A2}}, after passing to a subsequence and the changing the portrait, we can suppose that $v_0$ and $F(v_0)$ are in $T^X$. Suppose that $w_0$, $v_0$ and $F(v_0)$ are on a same path then, as the vertex $w_0$ is fixed (and the associated cover is the identity because it fixes the attaching points of the branches containing the fixed points and they have degree one), according to the annuli lemma we have $v_0=F(v_0)$ and thus from Corollary \ref{cvuutil}, the sequence converges uniformly to $[f_{v_0}]$ in $\rat_2$.
So they are in the closure of $\rat_2$ and the limit is a class of non polynomial covers with a triple fixed point in the cases {\bf DParb} and {\bf A2} or double fixed point in the case {\bf A3}. Conversely, we know that such elements in $\Rat_2$ are limits of dynamical systems between sphere covers with portrait ${\bf F}$ and it is clear that we obtain such dynamical covers as their limits.

Suppose by contradiction that $w_0$, $v_0$ and $F(v_0)$ are not on a same path then, using the annuli lemma in the non critical case, we prove that the vertex separating them is fixed and the corresponding cover is the identity and it follows that $v_0=F(v_0)$ which is absurd.

In configuration {\bf Stand}, the cover associated to $w_0$ is not the identity because the attaching point of the branch containing $\alpha$ doesn't map to itself. We deduce that the vertex $v_0$ is not fixed. If this cover is a dynamical limit of dynamical systems between spheres with portrait ${\bf F}$ and if there is a rescaling limit, then it has period more than $1$ and, according to Theorem \ref{thmkiw2}, the map $f_{w_0}$ has finite order so the multipliers at the fixed points $\beta$ and $\gamma$ are conjugated roots of unity.

\noindent{\textbf{Conclusion.}}


Remark that requiring to mark three non critical fixed points for the elements of $\Rat_{{\bf F},X}$ implies that
$\Rat_2\setminus \Rat_{{\bf F},X}\neq \emptyset$. However, it follows from the above discussion that all the elements of $\Rat_2\setminus \Rat_{{\bf F},X}$ can be identified in $\partial\Rat_{{\bf F},X}$. In addition, $\partial\Rat_{{\bf F},X}$ contains a blow up at the points corresponding to the rational map with a triple fixed point.








\end{document}